\documentclass{amsart}
\usepackage{amsmath}
\usepackage{amsfonts}
\usepackage{graphicx}
\usepackage{amssymb} 
\usepackage[all]{xy}

\newcommand{\Z}{\mathbb{Z}}
\newcommand{\Zq}{\mathbb{Z}_{q}}
\newcommand{\Zqn}{\mathbb{Z}_{q}^{n}}

\newcommand{\R}{\mathbb{R}}

\newcommand{\rmv}[1]{}
\newcommand{\simG}{\underset{\Gamma}{\sim}}
\newcommand{\simg}{\underset{\mathcal{G}}{\sim}}
\newcommand{\sima}{\underset{\mathcal{A}}{\sim}}

\makeatletter
\newcommand{\tpitchfork}{%
  \vbox{
    \baselineskip\z@skip
    \lineskip-.52ex
    \lineskiplimit\maxdimen
    \m@th
    \ialign{##\crcr\hidewidth\smash{$-$}\hidewidth\crcr$\pitchfork$\crcr}
  }%
}
\makeatother

\newtheorem{theorem}{Theorem}[section]
\newtheorem{lemma}[theorem]{Lemma}
\newtheorem{proposition}[theorem]{Proposition}
\newtheorem{corollary}[theorem]{Corollary}

\theoremstyle{definition}
\newtheorem{definition}[theorem]{Definition}
\newtheorem{remark}[theorem]{Remark}
\newtheorem{example}[theorem]{Example}
\newtheorem{acknowledgment}{Acknowledgement}

\newtheorem{notation}[theorem]{Notation}

\begin{document}

\title[On cube tilings of tori and classification of perfect codes in the maximum metric]{On cube tilings of tori and classification of perfect codes in the maximum metric}


\author{Claudio Qureshi}
\thanks{Work partially supported by FAPESP grants 2012/10600-2 and 2013/25977-7 and by CNPq grants 312926/2013-8 and 158670/2015-9.}
\address{Institute of Mathematics\\
State University of Campinas, Brazil}
\email{cqureshi@gmail.com}

\author{Sueli I.R. Costa}
\address{Institute of Mathematics\\
State University of Campinas, Brazil}
\email{sueli@ime.unicamp.br}

\maketitle

\begin{abstract}
We describe odd-length-cube tilings of the $n$-dimensional $q$-ary torus 
what includes $q$-periodic integer lattice tilings of $\mathbb{R}^{n}$. In the language of coding theory these tilings correspond to perfect codes with respect to the maximum metric. A complete characterization of the two-dimensional tilings is presented and a description of general matrices, isometry and isomorphism classes is provided in the linear case. Several methods to construct perfect codes from codes of smaller dimension or via sections are derived. We introduce a special type of matrices (perfect matrices) which are in correspondence with generator matrices for linear perfect codes in arbitrary dimensions. For maximal perfect codes, a parametrization obtained allows to describe isomorphism classes of such codes. We also approach the problem of what isomorphism classes of abelian groups can be represented by $q$-ary $n$-dimensional perfect codes of a given cardinality $N$.

\end{abstract}
\vspace{3mm}
{\scriptsize \textsc KEYWORDS } 
\\

{\scriptsize \textsc MATHEMATICS SUBJECT CLASSIFICATION}


\section{Introduction}

Despite the fact that cubes are among the simplest and important objects in Euclidean geometry, many interesting problems arise related to them. Sometimes complicated machineries from different areas of mathematics have had to be employed to solve some of the known results related to cubes and many basic problems still remain open, as it is pointed out in the survey \cite{Zong} on what is known about cubes. We focus here on a particular problem: lattice cube tilings. In 1906, an interesting conjecture was proposed by Minkowski 
\cite{Minkowski} while he was considering a problem in diophantine approximation. The Minkowski's conjecture states that in every tiling of $\R^n$ by cubes of the same length, where the centers of the cubes form a lattice, there exist two cubes that meet at a $(n-1)$-dimensional face. This conjecture was proved in 1942 by Haj\'{o}s \cite{Hajos}. A similar conjecture was proposed by Keller \cite{Keller} removing the restriction that the centers of the cubes have to form a lattice, however this more stronger version was shown to be true in dimensions $n\leq 6$ \cite{Perron}, false in dimensions $n\geq 8$ \cite{LS, Mackey} and it remains open in dimension 
$n=7$. A lot of variants and problems related to cube tilings have been considered \cite{KP1, KP2, KP3, SI, SIP} as well as application to other areas such as combinatorics, graph theory \cite{CS2}, coding theory \cite{LS2}, algebra \cite{Szabo}, harmonic analysis \cite{Kolountzakis} among others.\\


In this work we fix a triplet $(n,e,q)$ of positive integers and consider all the cube tilings of length $2e+1$ whose centers form an integer $q$-periodic lattice $\Lambda$ (i.e. an additive subgroup of $\R^n$ such that $q\Z^n \subseteq \Lambda \subseteq \Z^n$) and we classify these tilings with respect to the quotient group $\Lambda/q\Z^n$. Periodic cube tilings of $\R^n$ are in natural correspondence with cube tilings of the flat torus $\R^n/q\Z^n$, and the set $\Lambda/q\Z^n$ corresponds to the centers of the projected cubes in this torus. We consider the problem of describing all the quotient groups $\Lambda/ q\Z^n$ and classify them up to isometries and up to isomorphism. We also approach the problem of what group isomorphism classes are represented by these groups, obtaining structural information about these classes. One of our main motivation to consider these problems comes from coding theory, which gives us a natural framework as it is used in \cite{LS2}. The maximum (or Chebyshev) metric have been used in the context of coding theory and telecommunication mostly over permutation codes (rank modulation codes), some references include \cite{KLTT10,ST10,TS10}. Every $n$-dimensional $q$-ary linear code (i.e. a subgroup of $\Zq^n$) with respect to the maximum metric (induced from the maximum metric in $\Z^n$) is in correspondence with a $q$-periodic integer lattice via the so called Construction A \cite{CS}. This association has the property that each $e$-perfect code (i.e. a code in which every element of $\Z_q^n$ belong to exact one ball of radii $e$ centered at a codeword) corresponds to a cube tiling of length $2e+1$ of $\Z^n$, in such a way that its centers form a $q$-periodic integer lattice. We use the terminology of coding theory from some standard references such as \cite{MS, Lint} and \cite{CS}. Notation and results from coding theory and lattices, used in this paper, are included in Section 2. In Section 3 we study two-dimensional perfect codes, characterizing them (Corollary \ref{2DisC1orC2} and Theorem \ref{GeneratorForLPL2D}), describing what are the group isomorphism classes represented by the linear codes (Theorem \ref{Structure2DTh}) and providing a parametrization of such codes by a ring in such a way that isometry classes and isomorphism classes correspond with certain generalized cosets (Theorem \ref{parametrizationTheorem}). In Section 4 several constructions of perfect codes are derived, such as the linear-construction (Proposition \ref{LinearConstrProp}) and via sections (Proposition \ref{SectionConstrProp}) which allow us to describe some interesting families of perfect codes as those in Corollaries \ref{CyclicFamilyCor} and \ref{CyclicFamily2Cor} and to generalize some results from dimension two to arbitrary dimensions. In Section 5 we introduce a special type of matrices (perfect matrices) which characterize the generator matrices of perfect codes (Proposition \ref{PerfectGMProp}). For maximal codes (Definition \ref{MaximalDef}) a parametrization is given (Theorem \ref{GralParamTheoremP1}) and the induced parametrization of their isomorphism classes (Theorem \ref{GralParamTheoremP3}) as well as the number of isomorphism classes expressed in terms of certain generating function (Corollary \ref{NumberofIsomClassesCor}) are derived. We also describe the group isomorphism classes that can be represented by an $n$-dimensional $q$-ary code with packing radius $e$ (admissible structures) for the maximal case (Theorem \ref{AdmissibleTheorem}). Finally, in Section 6 we list further interesting research problems.

Some partial preliminary results stated in Sections 3 and 4 of this paper were presented in \cite{QC15}.

\section{Preliminaries, definitions and notations}

In this section we summarize results and notations which will be used in this paper. We use the language of coding theory to approach the geometric problem of tiling a flat torus by cubes.\\

Let $\Z_q=\{0,1,\ldots,q-1\}$ be the set of integers modulo $q$. Associated with $\Z_q$ we have a (non-directed) circular graph whose vertices are the elements of $\Z_q$ and the edges are given by $\{x,x+1\}$ for $x \in \Z_q$. The distance $d$ with respect to this graph is given by $d(x,y)=\min\{|x-y|,q-|x-y|\}$. For $p\in [1,\infty]$ and $x=(x_1,\ldots, x_n), y=(y_1,\ldots,y_n)\in\Zqn$, the $p$-Lee metric in $\Z_q^n$ is given by $d_p(x,y)= \left\{ \begin{array}{ll} \sqrt[p]{\sum_{i=1}^n d(x_i,y_i)^p} & \textrm{if } p \in [1,\infty), \\
\max_{i=1}^n d(x_i,y_i) & \textrm{if } p=\infty,
\end{array}   \right.$. We denote by $|x|_{p}=d_{p}(x,0)$ for $x \in \Z_q^n$. The case $p=1$ known as Lee metric is, apart from the Hamming metric, one of the most often used metric in coding theory due to several applications \cite{BBV,EY,RS,Schmidt}. Also of interest in this field are the cases $p=2$ (Euclidean distance) and $p=\infty$ (maximum or Chebyshev metric) which is our focus in this paper because of its association with cube packings. We denote the maximum metric simply by $d(x,y)$ and $B(x,r)$ is the ball centered at the point $x\in\Z_q^n$ and radius $r \in \mathbb{N}$ respect to this metric. \\

An $n$-dimensional $q$-ary code is a subset $C$ of $\mathbb{Z}_q^n$ and we assume here $\#C>1$. We refer to elements of $C$ as codewords and say that $C$ is a linear code when it is a subgroup of $(\mathbb{Z}_{q}^n, +)$. The minimum distance of $C$ is given by $\mbox{dist}(C)=\min\{d(x,y): x,y \in C, x\neq y\}$ and if $C$ is linear it is also given by $\mbox{dist}(C)=\min\{|x|_{\infty}: x \in C, x \neq 0\}$. The packing radius of $C$ is the greatest non-negative integer $e=e(C)$ (sometimes denoted by $r_p(C)$) such that the balls $B(c,e)$ are disjoint where $c$ runs over the codewords. An $n$-dimensional $q$-ary code with packing radius $e$ is called an $(n,e,q)$-code. The covering radius is the least positive integer $r_c=r_c(C)$ such that $\Z_{q}^n = \bigcup_{c\in C}B(c,r_c)$. A code $C$ is perfect when $r_p(C)=r_c(C)=e$, in this case we have a partition of the space into disjoint balls of the same radius $e$, that is $\Z_{q}^n = \biguplus_{c\in C}B(c,e)$. For perfect codes in the maximum metric we have $\mbox{dist}(C)=2e+1$. If $C \subseteq \mathbb{Z}_q^{n}$ is a perfect code with packing radius $e$ we have a map $f_C:\mathbb{Z}_{q}^n \rightarrow C$, called the correcting-error function of $C$, with the property $d(x,f_C(x))\leq e$ for $x\in \mathbb{Z}_q^n$. The set of all perfect $(n,e,q)$-codes is denoted by $PL^{\infty}(n,e,q)$ and its subset of linear codes is denoted by $LPL^{\infty}(n,e,q)$. The sphere packing condition states that for all $C \in PL^{\infty}(n,e,q)$ we have $\#C \times (2e+1)^n = q^n$ which implies that $q=(2e+1)t$ for some $t \in \mathbb{Z}^{+}$ and $\#C= t^n$. Conversely, if $q=(2e+1)t$ with $e,t \in \Z^{+}$ the code $C=(2e+1)\mathbb{Z}_q^{n}$ is an $(n,e,q)$-perfect code (called the cartesian code), so $LPL^{\infty}(n,e,q)\neq \emptyset$. From now on, we assume that $q=(2e+1)t$ where $e,t \in \Z^{+}$.\\


An $n$-dimensional lattice $\Lambda$ is a subset of $\R^n$ of the form $\Lambda = \nu_1 \Z + \nu_2 \Z + \ldots + \nu_n \Z$ where $\{\nu_1,\ldots,\nu_n\}$ is a basis of $\R^n$ (as $\R$-vector space). There is a strong connection between $n$-dimensional linear codes over $\Z_q$ and $q$-periodic integer lattices in $\R^n$. Associated with a linear $(n,e,q)$-code we have the lattice $\Lambda_{C}=\pi^{-1}(C)$ where $\pi: \Z^n \rightarrow \Z_q^n$ is the canonical projection (taking modulo q in each coordinate). This way of obtaining a lattice from a code is known as Construction A \cite{CS}. On the other hand, if $\Lambda \subseteq \R^n$ is a $q$-periodic integer lattice the image $\pi(\Lambda)$ is clearly an $n$-dimensional $q$-ary linear code and in fact, the map $\pi$ induces a correspondence between both sets. Considering the maximum metric in $\R^n$, the packing radius, covering radius and perfection for lattices with respect to this metric can be defined in an analogous way to how it was done for codes. In this sense, the above correspondence establish a bijection between $(n,e,q)$-perfect codes and $q$-periodic integer perfect lattices. We remark that the preservation of perfection by Construction A does not hold for $p$-Lee metrics in general. For example when $p=1$, the existence of an $e$-perfect $q$-ary code with respect to the Lee metric does not assure that the associated lattice is $q$-perfect if $q<2e+1$ (for example consider the binary code $C=\{(0,0,0),(1,1,1)\}$).\\

The above correspondence between codes and lattices allows us to use the machinery of lattices to approach problems in coding theory, specially those related with perfect codes. Next we introduce further notations.


\begin{notation} Let $A$ be a ring (in particular a $\Z$-module) and $a \in A$.\\
$\bullet$ $\mathcal{M}_{m\times n}(A)$ denotes the set of rectangular matrices $m\times n$ with coefficients in $A$. We identify $A^n$ with $\mathcal{M}_{1\times n}(A)$ and when $m=n$, we set $\mathcal{M}_n(A)=\mathcal{M}_{n\times n}(A)$.\\
$\bullet$ $\nabla_n(A)$ denotes the set of upper triangular matrices in $\mathcal{M}_{n}(A)$ and $\nabla_n(a,A)$ is the subset of $\nabla_n(A)$ whose elements in the principal diagonal are all equals to $a$. For $A=\mathbb{Z}$, we set $\nabla_n(a)=\nabla_n(a, \Z)$.\\
$\bullet$  For $x \in \mathbb{Z}^n$ we denote by $\overline{x}=x+q\Z^n \in \Z_q^n$. If $M \in \mathcal{M}_n(\Z)$ we denote by $\overline{M}$ the matrix obtained from $M$ taking modulo $q$ in each coordinate.\\
$\bullet$ Let $M \in \mathcal{M}_{m\times n}(A)$ whose rows are $M_1,\ldots,M_m$.
We denote by $\mbox{span}(M)=M_1 \Z + \ldots + M_m \Z$. In other words $\mbox{span}(M)$ is the subgroup of $A^n$ generated by the rows of $M$. \\
$\bullet$ We denote by $[n]=\{1,2,\ldots,n\}$ and by $\{e_i: i \in [n]\}$ the standard basis of $\R^n$ (or in general $e_i$ denotes the element of $A^n$ with a $1$ in the i-st coordinate and $0$ elsewhere).\\
$\bullet$ We denote by $S_{n}$ the group of bijection of the set $[n]$ (permutations).
\end{notation}

A standard way to define a code is from a generator matrix. For convenience, we will modify slightly the definition of generator matrix for a $q$-ary code.

\begin{definition}
Let $C$ be a $(n,e,q)$-code and $\Lambda_{C}$ be its associated lattice via Construction A. A {\sf generator matrix} for $C$ is a matrix $M \in \mathcal{M}_{n}(\Z)$ whose rows form a $\Z$-basis for the lattice $\Lambda_{C}$.
\end{definition}


\begin{remark}
Let $M$ be a generator matrix for a $q$-ary code $C$ and $\overline{M}$ the matrix obtained from $M$ taking modulo $q$ in each coordinate. Clearly the rows of $\overline{M}$ form a generating set for $C$ (as $\Z$-module), but the converse is false if $q\Z^n \not\subseteq \mbox{span}(M)$. In fact, if $M$ is a matrix such that $\overline{M}$ generates $C$ as $\Z$-module, we have that $M$ is a generator matrix for $C$ if and only if the matrix $qM^{-1}$ has integer coefficients.
\end{remark}


In addition to the problem of describing the set $LPL^{\infty}(n,e,q)$ we are also interested in describing isometry classes and isomorphism classes of such codes. We consider here only linear isometries of $\Zq^n$ (i.e. homomorphisms $f:\Zq^n \rightarrow \Zq^n$ which preserve the maximum norm). It is easy to see that for $q>3$, the isometry group $\mathcal{G}$ of $\Zq^n$ is given by $\mathcal{G}=\{\theta \eta_{a}: \theta \in S_n, a \in \Z_2^n\}$ where for $\theta(e_i)=e_{\theta(i)}$ and $\eta_a(e_i)=(-1)^{a_i}e_i$ for $1\leq i \leq n$. For $q\leq 3$ every monomorphism is an isometry.



\begin{notation}
Let $\phi: \mathcal{G} \times LPL^{\infty}(n,e,q) \rightarrow LPL^{\infty}(n,e,q)$ the action given by $(f,C)\mapsto f(C)$. We denote the quotient space by this action as $LPL^{\infty}(n,e,q)/\mathcal{G}$ and the orbit of an element $C$ by $[C]_{\mathcal{G}}$. If $[C_1]_{\mathcal{G}}=[C_2]_{\mathcal{G}}$ we say that $C_1$ and $C_2$ are geometrically equivalent and we denote this relation by $C_1 \simg C_2$.
\end{notation}

In order to describe isomorphism classes of codes we consider the equivalence relation in $LPL^{\infty}(n,e,q)$: $C_1 \sima C_2$ if there exists an isomorphism $f: \Zqn \rightarrow \Zqn$ such that $C_1 = f(C_2)$. When $C_1 \sima C_2$ we say that $C_1$ and $C_2$ are algebraically equivalent and we denote by $LPL^{\infty}(n,e,q)/\mathcal{A}$ the quotient set and by $[C]_{\mathcal{A}}$ the equivalence class of $C$ by this relation. It is important to remark that if $M$ is a generator matrix for a code $C \in LPL^{\infty}(n,e,q)$, then the group structure of $C$ is determined by the Smith normal form of $A=qM^{-1}$. Moreover, if the Smith normal form of $A$ is $D=\mbox{diag}(a_1,a_2,\ldots, a_n)$ then $C \simeq \Z_{a_1}\times \Z_{a_2}\times \ldots \times \Z_{a_n}$ and $\#C = \det(A)=q^n/\det(M)$ (Proposition 3.1 of \cite{AC}).\\

In Section 3 we parametrize isomorphism classes of perfect codes through certain generalized cosets of $\Z_{d}$ defined as follow.

\begin{definition}
Let $A$ be an abelian ring with unit and $A^{*}$ the multiplicative group of its invertible elements. A generalized coset is a set of the form $xC$ where $C<A^{*}$ (i.e. C is a multiplicative subgroup of $A^*$). We denote by $A/C=\{xC: x \in A\}$.
\end{definition}

\begin{remark}
Let $C<A^*$ and $x,y \in A$. If $xC \cap yC \neq \emptyset$ then $xC=yC$, so $C$ induces an equivalence relation in $A$ whose quotient set coincides with $A/C$.
\end{remark}

The $n$-dimensional $q$-ary flat torus $\mathcal{T}_q^n$ is obtained from the cube $[0,q]^n$ by identifying its opposite faces, see Figure \ref{toros}.

\begin{figure}[h]
  \begin{center}
    \includegraphics[width=0.5\textwidth,natwidth=471,natheight=366]{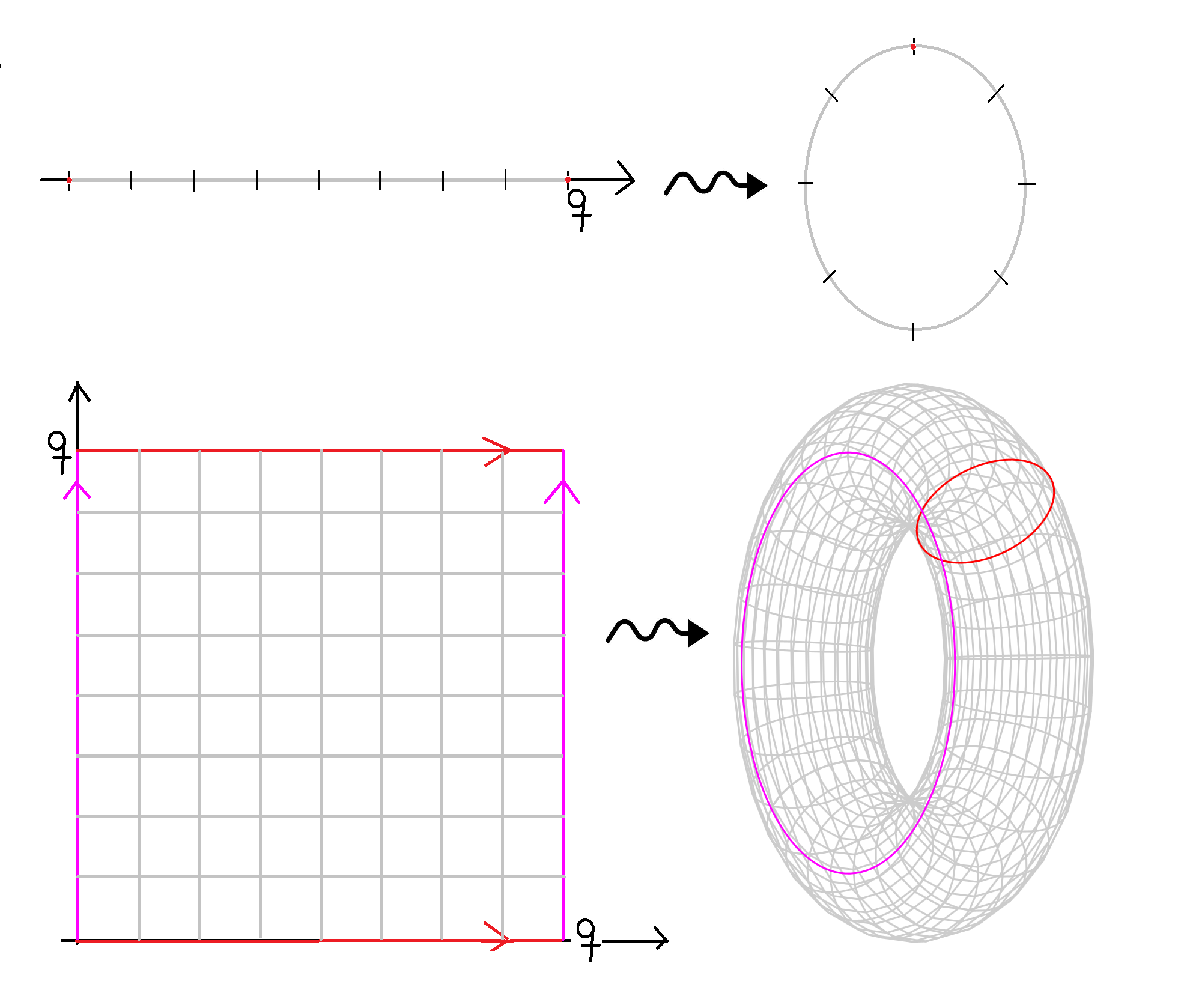}
  \caption{The torus in dimension one and two, obtained identifying opposite faces of the $n$-cube for $n=1,2$.}
  \label{toros}
  \end{center}
\end{figure}

It can also be obtained through the quotient $\mathcal{T}_{q}^n= \R^n / q \Z^n$, inheriting a natural group structure induced by this quotient. Given an invariant-by-translation metric $d$ in $\R^n$, every ball $B=B(0,r)$ in this metric have an associated polyomino $P_B \subseteq \mathbb{R}^n$ given by $$P_B= \biguplus_{x \in B\cap \Z^n} x+\left[-1/2,1/2\right]^n.$$ In this way, tiling $Z^n$ by translated copies of $B$ is equivalent to tiling $\mathbb{R}^n$ by translated copies of its associated polyomino $P_B$. When $P_B \subseteq \left[\frac{-q}{2},\frac{q}{2}\right]^n$, $q$-periodic tilings of $\mathbb{R}^n$ by translated copies of $P_B$ are in correspondence to tilings of $\mathcal{T}_{q}^n$ by translated copies of $\overline{P_B}$. This association provides an important geometric tool to study perfect codes over $\mathbb{Z}$. In the seminal paper of Golomb and Welch \cite{GW}, the authors use this approach based on polyominoes to settle several results on perfect Lee codes over large alphabets.\\ 

Since we consider the maximum metric, the polyominoes associated with balls in this metric correspond to cubes of odd length centered at points of $\mathbb{Z}^n$ (and only this type of cubes will be considered in this paper). The condition $q=(2e+1)t$ guarantees a correspondence between tiling of the torus $\mathcal{T}_q^n$ by cubes of length $2e+1$ and perfect codes in $LPL^{\infty}(n,e,q)$. 

\begin{definition}
An $n$-dimensional cartesian code is a code of the form $(2e+1)\mathbb{Z}_q^n$ for some $q\in \Z^{+}$ and $e \in \mathbb{N}$ such that $2e+1\mid q$. A linear $q$-ary code is a subgroup of $(\Z_q^n, +)$ (example in Figure \ref{ExCartesiano}). A cyclic $q$-ary code is a linear $q$-ary code which is cyclic as abelian group. The code $C=\Z_q^n$ is a perfect code and we refer to it as the trivial code.
\end{definition}

\begin{figure}[h]
  \begin{center}
    \includegraphics[width=0.4\textwidth,natwidth=471,natheight=366]{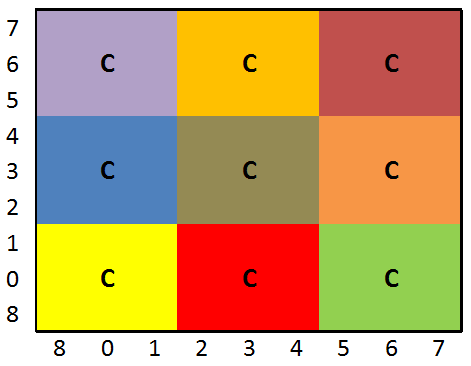}
  \caption{The cartesian code $3\mathbb{Z}_{9}^2 \in LPL^{\infty}(2,1,9)$ (codewords are marked with C)}
  \label{ExCartesiano}
  \end{center}
\end{figure}

\begin{definition}
We say that a code $C \in PL^{\infty}(n,e,q)$ is standard if there exists a canonical vector $e_i$ for some $i:1\leq i \leq n$ such that $C+(2e+1)e_i \subseteq C$. In this case we say also that $C$ is of type $i$, see Figure \ref{ExCiclico}.
\end{definition}

\begin{figure}[h]
  \begin{center}
    \includegraphics[width=0.4\textwidth,natwidth=477,natheight=363]{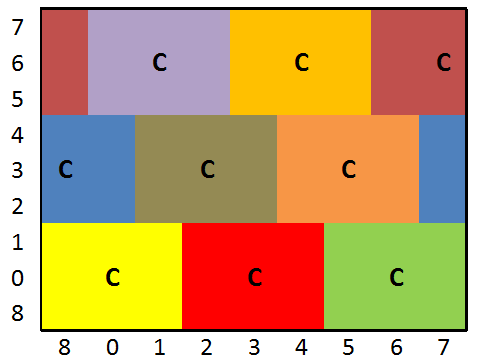}
  \caption{The cyclic perfect code $C=\langle (2,3) \rangle \in \mathbb{Z}_{9}^2 \in CPL^{\infty}(2,1,9)$ is a type $1$ code but is not a type $2$ code (codewords are marked with $C$).}
  \label{ExCiclico}
  \end{center}
\end{figure}

As we will see later (Remark \ref{NonStandardExampleRem}), a code could have no type or it could have more than one type (for example $n$-dimensional cartesian codes are of type $i$ for $1\leq i \leq n$).

The following theorem of Hajos \cite{Szabo}, known also as Minkowski's Conjecture, is of fundamental importance when we approach perfect codes in arbitrary dimensions.

\begin{theorem}[Minkowski-Hajos]\label{MinkowskiConjTh}
Every tiling of $\R^n$ by cubes of the same length whose centers form a lattice contains two cubes that meet in an $n-1$ dimensional face.
\end{theorem} 

\begin{corollary}\label{MinkoskiConjStandardCor}
Every linear perfect code $C \in LPL^{\infty}(n,e,q)$ is standard.
\end{corollary}



\section{Perfect codes in dimension two}

In this section we study two-dimensional perfect codes with respect to the maximum metric. First we describe the set of all (non-necessarily linear) $(2,e,q)$-perfect codes and how they can be obtained from an one-dimensional perfect code using horizontal or vertical construction. In particular we obtain that every two-dimensional perfect code is standard, what is not true for higher dimensions. This result in dimension $2$ is mentioned with no proof in \cite{Kisielewicz}. The proof presented here illustrates well the coding theory approach to be used in further results. Then we focus on the linear case providing generator matrices for perfect codes and describing isometry classes and isomorphism classes of the $(2,e,q)$-perfect codes.\\ 

The following result (whose proof is straightforward) characterizes the parameters for which there exists a perfect code with these parameters. 

\begin{proposition}\label{ExistencePropDim2}
A necessary and sufficient condition for the existence of an $n$-dimensional $q$-ary $e$-perfect code in the maximum metric is that $q=(2e+1)t$ for some integer $t>1$. Moreover, if this condition is satisfied, there exists a code in $LPL^{\infty}(n,q,e)$.
\end{proposition}


\begin{corollary}
There exists a non trivial perfect code over $\mathbb{Z}_q$ if and only if $q$ is neither a power of $2$ nor a prime number.
\end{corollary}

These results led us restricting to the case $q=(2e+1)t$ where $e\geq 0$ and $t>1$ are integers and we will maintain this assumption along this paper.


\subsection{Linear and non-linear two-dimensional perfect codes}

It is immediate to see that the only perfect codes $C$ in $PL^{\infty}(1,e,q)$ are of the form $a+(2e+1)\mathbb{Z}_q$ where $q=(2e+1)t$. If we fix a map $h:\mathbb{Z}_t \rightarrow \mathbb{Z}_q$ we can construct a two-dimensional $q$-ary perfect code as follows:\\
\noindent $\bullet$ (Horizontal construction) $C_1(a,h)=\{(h(k)+(2e+1)s,a+(2e+1)k):k, s \in \mathbb{Z}_t \}$.\\
$\bullet$ (Vertical construction) $C_2(a,h)= \{(a+(2e+1)k,h(k)+(2e+1)s):k,s \in \mathbb{Z}_t \}$.\\

The above construction give us $(2,e,q)$-codes of cardinality $t^2$ and minimum distance $d\geq 2e+1$ from which it is easy to deduce perfection. In fact we obtain $(2,e,q)$-perfect codes of type $1$ (if horizontal construction is used) or of type $2$ (if vertical construction is used). Moreover, every two-dimensional perfect code can be obtained in this way as we will see next. 

\begin{lemma}\label{eiLemma}
Let $\pi_i: \mathbb{Z}_q^n \rightarrow \mathbb{Z}_q$ be the canonical projection (i.e. $\pi(x_1,\ldots,x_n)=x_i$), $C\in PL^{\infty}(n,e,q)$, $f_C$ be its error-correcting function and $x$ be an element of $\mathbb{Z}_q^n$. Then,
\begin{itemize}
\item if $f_C(x)\neq f_C(x-e_i)$ then $\pi_i \circ f_C(x) = \pi_i(x)+e\cdot e_i$.
\item If $f_C(x)\neq f_C(x+e_i)$ then $\pi_i \circ f_C(x) = \pi_i(x)-e\cdot e_i$.
\end{itemize}
\end{lemma}

\begin{proof}
Let $c=f_C(x)$, $x_i=\pi_i(x)$ and $c_i=\pi_i(c)$. Denoting by $d$ the Lee metric in $\mathbb{Z}_q$, $f_C(x)=c$ implies $M_i=\max\{d(x_j,c_j): 1\leq j \leq n, j\neq i\}\leq e$ and $d(x_i,c_i)\leq e$. Since $f_C(x-e_i)\neq c$ we have that {$d(x-e_i,c)=\max\{M_i, d(x_i-1,c_i)\}\geq e+1$} and therefore $d(x_i-1,c_i)\geq e+1$. We obtain the inequalities $\|(c_i-x_i)\|\leq e$ and $\|(c_i-x_i)-1\|\geq$ \mbox{$e+1$} which imply $c_i-x_i=e$. The other case can be obtained from this considering the isometry $\eta_i$ of $\mathbb{Z}_q^n$ given by $\eta_i(x_1,\ldots,x_i,\ldots,x_n)=(x_1,\ldots,-x_i,\ldots,x_n)$.
\end{proof}

\begin{lemma}\label{standardDim2}
If $C \in PL^{\infty}(2,e,q)$ verifies $(2e+1)\mathbb{Z}_q \times \{0\} \subseteq C$ then $C$ is a standard code of type $1$.
\end{lemma}

\begin{proof}
Assume, to the contrary, that there is a codeword $c=(c_1,c_2)\in C$ such that $c+(2e+1)e_1=(c_1+2e+1,c_2)\not\in C$, and we take $c$ with this property such that $c_2$ is minimum. We claim that $c_2\geq 2e+1$. Indeed, if $0\leq c_2<2e+1$ and we express $c_1=(2e+1)k+r$ with $|r|\leq e$, then $((2e+1)k,e)$ belongs to both balls $B_{\infty}((2e+1)k,e)$ and $B_{\infty}(c,e)$ which is a contradiction, so $c_2\geq 2e+1$. We consider now $p=(c_1,c_2-(e+1))$, then $f_{C}(p+e_2)=c \neq f_{C}(p)$ and by Lemma \ref{eiLemma} we have $f_C(p)=(a,c_2-(2e+1))$ for some $a\in \mathbb{Z}_q$. We observe that $c_2-(2e+1)\geq 0$ and by the minimality of $c_2$ we have that $(a+(2e+1)k,c_2-(2e+1))\in C$ for $0\leq k <t$. Consider now $p'=(c_1+e+1,c_2-e)$ and express $c_1+e+1-a=(2e+1)v+w$ with $v,w \in \mathbb{Z}_q$ and $|w|\leq e$. Clearly, $f_C(p'-e_1)=c\neq f_C(p')$ and $f_C(p'-e_2)=(a+(2e+1)v,c_2-(2e+1))\neq f_C(p')$, thus, by Lemma \ref{eiLemma} we have $f_C(p')=p'+(e,e)=c+(2e+1)e_1$ which is a contradiction.
\end{proof}

By Minkowski-Hajos Theorem (Theorem \ref{MinkowskiConjTh}) every linear $n$-dimensional prefect code in the maximum metric is standard, but in dimension two this is also true for all perfect codes.

\begin{proposition} \label{2Disstandard}
Every two-dimensional perfect code in the maximum metric is standard.
\end{proposition}

\begin{proof}
Let $C\in PL^{\infty}(2,e,q)$ with $q=(2e+1)t$ and $t>1$ an integer. Let we suppose that $C$ is not of type 2, so there exists a codeword $c\in C$ such that $c+(2e+1)e_2 \not\in C$. Composing with a translation if it is necessary we can assume $c=0$. Consider $p=(0,e+1)$, by Lemma \ref{eiLemma} we have $f_C(p)=(a,2e+1)$ with $|a|\leq e$ and $a\neq 0$. Composing with the isometry $(x,y)\mapsto (-x,y)$ if necessary we may assume $0<a\leq e$. We consider the following statement: $\{(p_h=(2e+1)h,2e+1),q_h=(a+(2e+1)h,2e+1)\}\subseteq C$, which is valid for $h=0$ (from above). Assume that this property holds for a fixed $h$, $0\leq h < t$ and consider $p=((2e+1)h+e+1,e)$. This point verifies $f_C(p-e_1)=p_h\neq f_C(p)$ and $f_C(p+e_2)=q_h\neq f_C(p)$, so by Lemma \ref{eiLemma} we have $f_C(p)=p+(e,-e)=((2e+1)(h+1),0)=p_{h+1} \in C$. Now consider $p'=(a+(2e+1)h+e+1,e+1)$ which verifies $f_C(p'-e_1)=q_h\neq f_C(p')$ and $f_C(p'-e_2)=p_{h+1}\neq f_C(p')$, so by Lemma \ref{eiLemma} we have $f_C(p')=p'+(e,e)=(a+(2e+1)(h+1),2e+1)=q_{h+1}\in C$. By induction we have that $(2e+1)\mathbb{Z}_q \times \{0\} \subseteq C$ and by Lemma \ref{standardDim2} our code $C$ is standard.
\end{proof}

\begin{corollary} \label{2DisC1orC2}
Every two-dimensional perfect code $C$ in the maximum metric is of the form $C=C_1(a,h)$ or $C=C_2(a,h)$ for some $a\in \mathbb{Z}_q$ and some function $h:\mathbb{Z}_t \rightarrow \mathbb{Z}_q$.
\end{corollary}

\begin{remark}\label{NonStandardExampleRem}
Proposition \ref{2Disstandard} cannot be generalized to higher dimensions. For example the code $C=\{(0,0,0),(5,0,0),(1,0,5),(6,0,5),(1,5,0),(6,5,1),$ $(1,5,5),$ $(6,5,6)\} \in PL^{\infty}(3,2,10)$ is a three-dimensional non-standard perfect code.
\end{remark}


\begin{corollary}
The number of $(2,e,q)$-perfect codes is $(2e+1)^2\left( 2(2e+1)^{t-1}-1 \right)$.
\end{corollary}

\begin{proof}
We consider the sets $L=PL^{\infty}(2,e,q), L^0 = \{C \in L: 0\in C \}$ and $L_{i}^{0}= \{C \in L^0: C \textrm{ is of type }i\}$ for $i=1,2$. The map $L \twoheadrightarrow L^0$ given by $C \mapsto C-f_C(0)$ is $(2e+1)^2$ to $1$, so $\#L = (2e+1)^2 \#L^0$. By Proposition \ref{2Disstandard}, $L^{0}=L_1^{0}\cup L_2^{0}$ and considering the involution $L_1^{0}\rightarrow L_2^{0}$ given by $C \mapsto \theta(C)$ where $\theta(x,y)=(y,x)$, which has exactly one fixed point (given by $(2e+1)\Z_q^n$) we have $\#L^{0}=2\#L_2^{0}-1$. Finally, codes in $L_2^{0}$ are univocally determined by a function $h: \Z_t \rightarrow \Z_{2e+1}$ verifying $h(0)=0$, thus $\L_2^{0}=(2e+1)^{t-1}$ and so $\#L=(2e+1)^2\left( 2(2e+1)^{t-1}-1 \right)$.
\end{proof}

\subsection{Generator matrices and admissible structures for two-dimensional perfect codes} In this part we provide generator matrices for linear perfect codes and we describe all two-dimensional cyclic perfect codes in the maximum metric. A description of which group structure can be represented by two-dimensional linear perfect codes is given.


\begin{notation}
We denote by $LPL^{\infty}(2,e,q)_{o}$ the set of $(2,e,q)$-perfect codes of type $2$.
\end{notation}

By Proposition \ref{2Disstandard}, every two-dimensional perfect code is of type $1$ or is of type $2$ and the isometry $\pi(x,y)=(y,x)$ induces a correspondence between the codes of type $1$ and codes of type $2$. So, without loss of generality we can restrict our study to type $2$ perfect codes.

\begin{theorem}\label{GeneratorForLPL2D}
Let $q=(2e+1)t$ with $t>1$, $d_1=\gcd(2e+1,t)$ and $h_1= \frac{2e+1}{d_1}$. Every integer matrix of the form $M= \left( \begin{matrix}
2e+1 & kh_1 \\ 0 & 2e+1 \end{matrix} \right)$ with $k \in \Z$ is the generator matrix of some type $2$ perfect code $C \in LPL^{\infty}(2,e,q)_{o}$. Conversely, every type $2$ perfect code $C \in LPL^{\infty}(2,e,q)_{o}$ has a generator matrix of this form.
\end{theorem}

\begin{proof}
Let $M= \left( \begin{matrix} 2e+1 & kh_1 \\ 0 & 2e+1 \end{matrix} \right)$ with $k \in \Z$. Since $qM^{-1}=\left( \begin{matrix}
t & -k\frac{t}{d_1} \\ 0 & t \end{matrix} \right)$ has integer coefficient, $M$ is the generator matrix of the $q$-ary code $C=\langle \overline{c_1}, \overline{c_2} \rangle \subseteq \Z_q^2$ where $c_1=(2e+1,kh1)$ and $c_2=(0,2e+1)$. Every codeword is of the form $\overline{c}=x\overline{c_1}+y\overline{c_2}$ with $x,y \in \Z$. Since $||\overline{c}||_{\infty}=\max\{|\overline{(2e+1)x}|_{1}, |\overline{kh_1x+(2e+1)y|_{1}}\}$, the inequality $||\overline{c}||_{\infty}<2e+1$ implies $\overline{c}=0$, thus the minimum distance of 
$C$ is $\mbox{dist}(C)\geq 2e+1$. In addition, the cardinality of $C$ is $\#C=q^2/\det(M)=t^2$, so by the sphere packing condition the code $C$ is perfect with packing radius $e$ and it is of type $2$ because is linear and $\overline{c_2}=(2e+1)\overline{e_2}\in C$. To prove the converse we consider a code $C \in LPL^{\infty}(2,e,q)_{o}$, since $C$ is linear then $0\in C$ and $C=C_2(0,h)$ for some $h: \Z_t \rightarrow \Z_q$. In particular, $C$ has two codewords $c_1=(\overline{2e+1},\overline{y_1})$ and $c_2=(\overline{0},\overline{2e+1})$ where $y_1\in\Z$ is such that $\overline{y_1}=h(1)\in\Z_q$. Let $ty_1=(2e+1)s+r$ with $s,r\in\Z$ and $0\leq r < 2e+1$. By linearity $tc_1-sc_2=(\overline{0},\overline{r})\in C$ which has minimum distance $2e+1$, so $r=0$ and $y_1=h_1k$ for some integer $k$. The code $C'$ generated by $c_1$ and $c_2$ has generator matrix $M= \left( \begin{matrix} 2e+1 & kh_1 \\ 0 & 2e+1 \end{matrix} \right)$, therefore by the first part, the code $C'$ generated by $c_1$ and $c_2$ is a $(2,e,q)$-perfect code which is contained in $C$, so $C'=C$.
\end{proof}

\begin{notation}
We denote by $LC_q(e,k)$ the $q$-ary perfect code whose generator matrix is given by $ \left( \begin{matrix} 2e+1 & kh_1 \\ 0 & 2e+1 \end{matrix} \right)$.
\end{notation}

\begin{remark}\label{GeneratorForLPL2DRemark}
Replacing the first row by its sum with an integer multiple of the second row if it were necessary, we can always suppose that the number $k$ in the statement of Theorem \ref{GeneratorForLPL2D} verify $0\leq k < d_1$. In fact, it is possible replace $k$ by any integer congruent to $k$ modulo $d_1$ with this elementary operation in rows, so $LC_q(e,k)=LC_q(e,k_0)$ if $k\equiv k_0 \pmod{d_1}$.
\end{remark}

Now we approach the problem of what group isomorphism classes are represented by linear $(2,e,q)$-perfect codes ({\it admissible structures}). By the sphere packing condition, if $C \in LPL^{\infty}(2,e,q)$ then $\#C=t^2$. The structure theorem for finitely generated abelian groups \cite[p.~338]{Fraleigh} in this case establishes that $C\simeq \mathbb{Z}_{t/d}\times \mathbb{Z}_{dt}$ for some $d|t$ (where $d$ determines the isomorphism class of $C$). In this way, the question of what isomorphism classes are represented by two-dimensional perfect codes in the maximum metric is equivalent to determining for what values of $d|t$ there exists $C\in LPL^{\infty}(2,e,q)$ such that $C\simeq \mathbb{Z}_{t/d}\times \mathbb{Z}_{dt}$.

\begin{lemma}\label{OtherGeneratorforLPL2DLemma}
Let $q=(2e+1)t, d_1=\gcd(2e+1,t), h_1=\frac{2e+1}{d_1}, d_2=\gcd(d_1,k), h_2=\frac{d_1}{d_2}, k_1= \frac{k}{d_2}$ and $k'\in \Z$ such that $k_1 k \equiv 1 \pmod{h_2}$. Then $N=\left( \begin{matrix}
(2e+1)h_2 & 0 \\ (2e+1)k' & h_1 d_2 \end{matrix}  \right)$ is a generator matrix for $LC_q(e,k)$.
\end{lemma}

\begin{proof}
Let $M$ be the generator matrix for $LC_q(e,k)$ given in Theorem \ref{GeneratorForLPL2D} and $U= \left( \begin{matrix}
h_2 & -k_1 \\ k' & \frac{1-k_1k'}{h_2}
\end{matrix}  \right)$. Since $\det(U)=1$ and $UM=N$ we have that $N$ is also a generator matrix for $LC_q(e,k)$. 
\end{proof}

\begin{theorem}\label{Structure2DTh}
Let $q=(2e+1)t$, $k$ be an integer and $h_2=\frac{\gcd(2e+1,t)}{\gcd(2e+1,t,k)}$.
\begin{itemize}
\item[(i)] $LC_q(e,k)\simeq \Z_{{t}/{h_2}}\times \Z_{th_2}$ (isomorphic as groups).
\item[(ii)] There exists $C \in LPL^{\infty}(2,e,q)$ such that $C \simeq \Z_{{t}/{d}}\times \Z_{td}$ if and only if $d \mid \gcd(2e+1,t)$.
\end{itemize}
\end{theorem}

\begin{proof}
To prove (i) we consider the homomorphism $T:\Z^2 \rightarrow \Z_q^2$ given by $T(x)=x\overline{N}$, where $N$ is as in Lemma \ref{OtherGeneratorforLPL2DLemma}. We have that $\ker(T)=\frac{t}{h_2}\Z \times th_2\Z$ and by the referred lemma $\mbox{Im}(T)=LC_q(e,k)$, so (i) follows from the First group isomorphism theorem \cite[p.~307]{Fraleigh}. To prove (ii) we observe that for every $k$ we have $h_2 \mid \gcd(2e+1,t)$ and for $d \mid d_1$ where $d_1=\gcd(2e+1,t)$, then $LC_q(e,\frac{d_1}{d}) \simeq \Z_{{t}/{d}}\times \Z_{td}$.
\end{proof}

\begin{corollary}
There exists a two-dimensional perfect code $C\simeq \mathbb{Z}_{a}\times \mathbb{Z}_{b}$ with $a\mid b$ if and only if $ab$ is a perfect square and $b/a$ is an odd number.
\end{corollary}

\begin{proof}
($\Rightarrow$) By Theorems \ref{GeneratorForLPL2D} and \ref{Structure2DTh}, if $\Z_{a}\times \Z_{b}\simeq C$ with $a|b$ for some perfect code $C$, then there exists integers $t,h_2$ and $e$ such that $a=\frac{t}{h_2},b=th_2$ and $h_2\mid 2e+1$ (in particular $h_2$ is odd), thus $ab=t^2$ is a perfect square and $\frac{b}{a}=h_2^2$ is odd. \\
($\Leftarrow$) Let $ab=t^2$ and $b=as$ with $s$ odd. Since $a^2s=t^2$ we have $s=(2e+1)^2$ and $(2e+1)a=t$. Defining $q=(2e+1)t$, by Theorem \ref{Structure2DTh} we have $LC_q(e,1)\simeq \Z_a \times \Z_b$.
\end{proof}

\begin{corollary}
Let $C \in LPL^{\infty}(2,e,q)$ with $q=(2e+1)t$. Then, $C\simeq \mathbb{Z}_{t}\times \mathbb{Z}_{t} \Leftrightarrow C$ is the cartesian code $C=(2e+1)\mathbb{Z}_{q}^2$.
\end{corollary}

\begin{proof}
By Theorem \ref{GeneratorForLPL2D} and Remark \ref{GeneratorForLPL2DRemark} every code is of the form $C=LC_{q}(e,k)$ for some $k$ with $0\leq k < d_1$ and by Theorem \ref{Structure2DTh} we have $C \simeq \mathbb{Z}_{t} \times \mathbb{Z}_{t} \Leftrightarrow h_2=1 \Leftrightarrow d_1=d_2 \Leftrightarrow d_1 \mid k \Leftrightarrow 2e+1\mid kh_1 \Leftrightarrow k=0 \Leftrightarrow C=(2e+1)\mathbb{Z}_{q}^2$.
\end{proof}

\begin{corollary}\label{ExistsNonStandard2DCor}
There exists a linear two-dimensional $q$-ary perfect code $C$ that is non-cartesian if and only if $q=p^2a$ where $p$ is an odd prime number and $a$ is a positive integer.
\end{corollary}

\begin{proof}
By Theorem \ref{Structure2DTh} part (ii), there exists a $q$-ary non-cartesian perfect code if and only if $q=(2e+1)t$ for some integers $e$ and $t$ such that $\gcd(2e+1,t)>1$. This last condition is equivalent to $2e+1=pm$ and $t=pn$ for some odd prime $p$ and $m,n \in Z^{+}$, thus $q=p^2a$ where $p$ is an odd prime and $a$ is a positive integer. 
\end{proof}

\begin{example}
The first value of $q$ for which there exists a two-dimensional $q$-ary perfect code that is neither cartesian nor cyclic is $3^2\cdot 2$. An example of such code has generators $\{(0,9),(1,3)\}\subseteq \mathbb{Z}_{18}^2$, see Figure \ref{ExZ2xZ18}.
\end{example}

\begin{figure}[h]
  \begin{center}  
    \includegraphics[width=0.65\textwidth,natwidth=452,natheight=342]{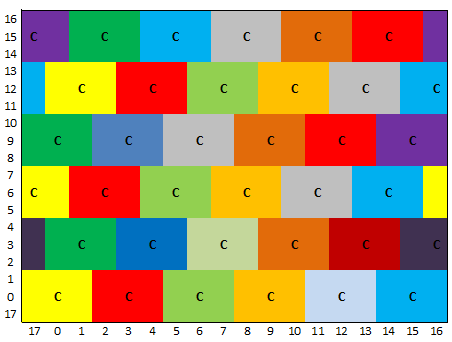}
  \caption{For $q=3^2\cdot2$, the perfect code $C=\langle (0,9),(1,3)\rangle \subseteq \mathbb{Z}_{18}^2$ is isomorphic to $\mathbb{Z}_2 \times \mathbb{Z}_{18}$.}
  \label{ExZ2xZ18}
  \end{center}
\end{figure}

\begin{corollary}\label{qCyclic2DCor}
There exists a two-dimensional cyclic $q$-ary perfect code if and only if $q=p^2a$ where $p$ is an odd prime number and $a$ is an odd positive integer.
\end{corollary}

\begin{proof}
By Theorem \ref{Structure2DTh} part (ii), there exists a $q$-ary cyclic perfect code if and only if $q=(2e+1)t$ for some integers $e$ and $t$ such that $\gcd(2e+1,t)=t>1$. This last condition is equivalent to $2e+1=mt$ for some odd integer $m$, thus $q=mt^2$ where $a$ is an odd integer and $t>1$ which is equivalent to $q=ap^2$ where $a$ is an odd integer and $p$ is an odd prime number.
\end{proof}

\begin{corollary}
Let $q=(2e+1)t$. There exists a cyclic code in $LPL^{\infty}(2,e,q)$ if and only if $t\mid 2e+1$. Under this condition $LC_q(e,k)$ is cyclic if and only if $\gcd(k,t)=1$.
\end{corollary}

\begin{proof}
By Theorem \ref{Structure2DTh} part (ii), there exists a cyclic code in $LPL^{\infty}(2,e,q)$ if and only if $\gcd(2e+1,t)=t$ if and only if $t\mid 2e+1$. In this case, by Theorem \ref{Structure2DTh} part (i), we have $LP_{q}(e,k)\equiv \Z_{t^2} \Leftrightarrow h_2=t \Leftrightarrow$ and $\gcd(2e+1,t,k)=\gcd(t,k)=1$.
\end{proof}

\subsection{Isometry and isomorphism classes of two-dimensional perfect codes}
Since every linear $(2,e,q)$-perfect code is isometric to an type $2$ linear perfect code we can restrict to $LPL^{\infty}(2,e,q)_{o}$. Our main result here is a parametrization of the set $LPL^{\infty}(2,e,q)_{o}$ by the ring $\Z_{d_1}$ (where $d_1=\gcd(2e+1,t)$) in such a way that isometry classes and isomorphism classes correspond to certain generalized cosets.\\

The following lemma can be obtained from the Chinese remainder theorem.

\begin{lemma}\label{InvisibleLemma} There exist $u\in \mathbb{Z}_d^{*}$ such that $a \equiv u b \pmod{d} \Leftrightarrow \gcd(a,d)=\gcd(b,d)$.
\end{lemma}

\begin{theorem}\label{parametrizationTheorem}
Let $q=(2e+1)t, d_1=\gcd(2e+1,t)$ and $h_1= \frac{2e+1}{d_1}$. We have the parametrization (bijection):
$$ \psi: \mathbb{Z}_{d_1} \rightarrow LPL^{\infty}(2,e,q)_o$$
$$  k+d_1\mathbb{Z} \mapsto LC_q(e,k), \qquad$$ which induces the parametrizations:
$$ \psi_\mathcal{G}: \frac{\mathbb{Z}_{d_1}}{\{1,-1\}} \rightarrow LPL^{\infty}(2,e,q)_o/\mathcal{G}$$
$$  k\cdot\{1,-1\}\mapsto [\psi(k)]_{\mathcal{G}}, \qquad\quad$$
and
$$ \psi_\mathcal{A}: \frac{\mathbb{Z}_{d_1}}{\mathbb{Z}_{d_1}^{*}} \rightarrow LPL^{\infty}(2,e,q)_o/\mathcal{A}$$
$$  k\cdot\mathbb{Z}_{d_1}^{*}\mapsto [\psi(k)]_{\mathcal{A}}, \qquad\quad$$
\end{theorem}

\begin{proof}
By Theorem \ref{GeneratorForLPL2D} and Remark \ref{GeneratorForLPL2DRemark} the map $\psi$ is well defined and is a surjection, so it remains to prove that $LC_q(e,k_1)=LC_q(e,k_2) \Leftrightarrow k_1 \equiv k_2 \pmod{d_1}$. Since both codes have the same cardinality $t^2$, we have
$$LC_q(e,k_1)=LC_q(e,k_2) \Leftrightarrow LC_q(e,k_1) \subseteq LC_q(e,k_2) \Leftrightarrow (h_1k_1,2e+1)\in LC_q(e,k_2)$$ $$\Leftrightarrow \exists x,y, \in \mathbb{Z} : \left\{ \begin{array}{l}
(2e+1)x+h_1 k_2 y \equiv h_1 k_1 \pmod{q} \\ (2e+1)y\equiv 2e+1 \pmod{q}
\end{array} \right.$$ $$\Leftrightarrow \exists x,y, \in \mathbb{Z} : \left\{ \begin{array}{l}
y\equiv1\pmod{t}\\ d_1x+k_2y\equiv k_1\pmod{td_1}
\end{array} \right. \Rightarrow \exists y \in \mathbb{Z}: \left\{\begin{array}{l}
y\equiv 1 \pmod{d_1} \\ k_2 y \equiv k_1 \pmod{d_1}
\end{array} \right.$$  which implies $k_1 \equiv k_2 \pmod{d_1}$.\\

Let $\eta_i: \mathbb{Z}_q^2 \rightarrow \mathbb{Z}_q^2$ for $i=1,2$ given by $\eta_1(x,y)=(-x,y)$ and $\eta_2(x,y)=(x,-y)$. We have that $\eta_1(LC_q(e,k))=\langle (-(2e+1),0),(-kh_1,2e+1) \rangle = LC_q(e,-k)$ and the same is valid for $\eta_2$, thus $[\psi(k)]_{\mathcal{G}}=\{\psi(-k),\psi(k)\}$ and so $\psi_{\mathcal{G}}$ is well defined and is a bijection. By Theorem \ref{Structure2DTh} we have $\psi(k)=LC_q(e,k)\simeq \mathbb{Z}_{t/h_2}\times \mathbb{Z}_{th_2}$ where $h_2 = \frac{d_1}{\gcd(d_1,k)}$, therefore $[\psi(k_1)]_{\mathcal{A}}=[\psi(k_2)]_{\mathcal{A}} \Leftrightarrow \gcd(k_1,d_1)=\gcd(k_2,d_1)$ and so $k_1 \equiv u k_2 \pmod{d_1}$ for some $u\in\mathbb{Z}$ with $gcd(u,d_1)=1$ (Lemma \ref{InvisibleLemma}), which is equivalent to $k_1 \mathbb{Z}_{d_1}^{*}=k_2\mathbb{Z}_{d_1}^{*}$.
\end{proof}

\begin{example} Let $p>2$ be a prime number and we take $q=p^2$ and $e\geq1$ such that $2e+1=p$. In this case $d_1=p$ and we have exactly $p$ codes in $LPL^{\infty}(2,p,p^2)$ given by $LC_{p^2}(p,k)$ for $0\leq k <p$, where the code $LC_{p^2}(p,k)$ has generator matrix $M_k =\left( \begin{array}{cc}
p & 0 \\ k & p \end{array}  \right) \in \mathcal{M}_{2\times 2}(\mathbb{Z}_{p^2})$. There exist exactly $\frac{p+1}{2}$ of such perfect codes up to isometry, given by $LC_{p^2}(p,k)$ for $0\leq k \leq \frac{p-1}{2}$. Since $p$ is prime, we have $\mathbb{Z}_{p}=\{0\}\uplus \mathbb{Z}_p^{*}$, so there exist exactly $2$ perfect codes in $LPL^{\infty}(2,p^2,p)$ up to isomorphism, one of which is the cartesian code (which corresponds to $k=0$) and the other is $LC_{p^2}(p,1)$ (which is isomorphic to $LC_{p^2}(p,k)$ for $1<k<p$).
\end{example}

\begin{corollary}
The set $\{LC_q(e,k): 0\leq k \leq \frac{d_1-1}{2}\}$ is a set of representative of $LPL^{\infty}(2,e,q)/\mathcal{G}$ and the set $\{LC_q(e,k): \gcd(k,d_1)=1\textrm{, } 1\leq k \leq d_1 \}$ is a set of representative of $LPL^{\infty}(2,e,q)/\mathcal{A}$.
\end{corollary}

\begin{corollary}\label{NumberOfCodes2DCor}
There exist exactly $d_1=\gcd(2e+1,t)$ codes in $LPL^{\infty}(2,e,q)$ where $q=(2e+1)t$. There exist exactly $\frac{d_1+1}{2}$ of such codes up to isometry and there exist exactly $\sigma_0(d_1)$ of such codes up to isomorphism where $\sigma_0$, as usual, denotes the number-of-divisors function.
\end{corollary}

\begin{proof}
The first two assertion are immediate. For the third assertion we use that $\mathbb{Z}_{d_1}=\biguplus_{d\mid d_1}d \mathbb{Z}_{d_1}^{*}$ and use Theorem \ref{parametrizationTheorem}.
\end{proof}

\section{Constructions of perfect codes in arbitrary dimensions}

In this section we give some constructions of perfect codes in the maximum metric from perfect codes in smaller dimensions. We also present a section construction which plays an important role in the next section.\\

The simplest way to obtain perfect codes is using cartesian product. Using the sphere packing condition we obtain the following proposition.

\begin{proposition}[Cartesian product construction] \label{CartesianProductProp}
If $C_1 \in PL^{\infty}(n_1,e,q)$ and $C_2 \in PL^{\infty}(n_1,e,q)$ then $C_1\times C_2 \in PL^{\infty}(n_1+n_2,e,q)$. This construction preserves linearity.
\end{proposition}


\begin{corollary}
There exists a linear non-cartesian $n$-dimensional $q$-ary perfect code if and only if $q=p^2a$ where $p$ is an odd prime number and $a$ is a positive integer.
\end{corollary}


\begin{corollary}\label{StructureFromCartProdCor}
If $q=(2e+1)t$ and $d_1,d_2,\ldots,d_k$ are divisors (not necessarily distinct)
of $\gcd(2e+1,t)$, there exists a code $C\in LPL^{\infty}(2k,e,q)$ such that $$C \simeq \mathbb{Z}_{\frac{t}{d_1}}\times \mathbb{Z}_{\frac{t}{d_2}}\times \ldots \mathbb{Z}_{\frac{t}{d_k}} \times \mathbb{Z}_{d_1t}\times \mathbb{Z}_{d_2t}\times \ldots \mathbb{Z}_{d_kt}$$ and a code $C\in LPL^{\infty}(2k+1,e,q)$ such that $$C \simeq \mathbb{Z}_{\frac{t}{d_1}}\times \mathbb{Z}_{\frac{t}{d_2}}\times \ldots \mathbb{Z}_{\frac{t}{d_k}} \times \mathbb{Z}_{t}\times \mathbb{Z}_{d_1t}\times \mathbb{Z}_{d_2t}\times \ldots \mathbb{Z}_{d_kt}.$$
\end{corollary}


\begin{remark}
There are other linear perfect codes whose group structure is not of the form given in Corollary \ref{StructureFromCartProdCor} (for example those in Corollary \ref{CyclicFamilyCor}).
\end{remark}

The next construction is specific for linear codes, this allows to construct a linear perfect $q$-ary code from other codes of smaller dimensions.

\begin{notation}\label{tALaMenos1Not}
If $H$ is a subgroup of an abelian group $G$ and $t \in \mathbb{Z}^+$, we denote by $t^{-1}H= \{g \in G: tg \in H\}$.
\end{notation}

\begin{remark}
With the above notation, $t^{-1}H$ is a subgroup of $G$ that contains $H$.
\end{remark}

\begin{proposition}[Linear construction]\label{LinearConstrProp}
If $C \in LPL^{\infty}(n,e,q)$ with $q=(2e+1)t$ and $x \in t^{-1}C$, then $\widetilde{C}= C \times \{0\}  +(x,2e+1)\mathbb{Z} \in LPL^{\infty}(n+1,e,q)$.
\end{proposition}

\begin{proof}
Since $tx \in C$ every codeword $v \in \widetilde{C}$ can be written as $v=(c+xk,(2e+1)k)$ with $c \in C$ and $0\leq k <t$ and we have \begin{equation}\label{infEq} \|(c+xk,(2e+1)k)\|_{\infty}= \max\{\|c+xk\|_{\infty}, \|(2e+1)k\|_{\infty}\}.\end{equation}
If $k=0$, then $\|(c+xk,(2e+1)k)\|_{\infty}=\|c\|_{\infty}\geq 2e+1$ if $c\neq 0$ (because $C$ have packing radius $e$).
If $0<k<t$, then $\|(2e+1)k\|_{\infty}\geq 2e+1$ and by \eqref{infEq} we have $\|(c+xk,(2e+1)k)\|_{\infty}\geq 2e+1$. We conclude that $C$ has packing radius at least $e$. We want to calculate the cardinality of $C$, that is \begin{equation}\label{CcardinalityEq}\#C = \frac{\# C \times \{0\} \cdot \#(x,2e+1)\mathbb{Z}}{\# C\times \{0\} \cap (x,2e+1)\mathbb{Z}}. \end{equation}
We have $\# C \times \{0\}= \#C = t^n$. Let $\theta$ the additive order of $tx$ in $\mathbb{Z}_q^n$ (i.e. the least positive integer $\theta$ such that $\theta t x =0$). It is straightforward to check that the order of $(x,2e+1)$ in $\mathbb{Z}_q^{n+1}$ is $t\theta$ and that $C\times \{0\} \cap (x,2e+1)\mathbb{Z}=(tx,0)\mathbb{Z}$. Using equation \eqref{CcardinalityEq} we have $\#C = \frac{t^{n}\cdot t\theta}{\theta}=t^{n+1}$ and by the sphere packing condition the code $\widetilde{C}\subseteq \mathbb{Z}_q^{n+1}$ is perfect with packing radius $e$.
\end{proof}

\begin{corollary}\label{CyclicFamilyCor}
If $q=(2e+1)t$ with $t^{n-1}\mid 2e+1$ and $n\geq 1$, then the $q$-ary cyclic code $$\mathcal{C}_{n,e,q}= \left\langle \left(\frac{2e+1}{t^{n-1}},\frac{2e+1}{t^{n-2}},\ldots,\frac{2e+1}{t},2e+1\right) \right\rangle \in LPL^{\infty}(n,e,q).$$
\end{corollary}

\begin{proof}
We denote by $p_n=\left(\frac{2e+1}{t^{n-1}},\frac{2e+1}{t^{n-2}},\ldots,\frac{2e+1}{t},2e+1\right)\in \mathbb{Z}_{q}^{n}$ and proceed by induction. For $n=1$ it is clear. If $C_{n,e,q}\in LPL^{\infty}(n,e,q)$ holds for some $n\geq 1$, we apply the linear construction with $x=\left(\frac{2e+1}{t^{n}},\frac{2e+1}{t^{n-2}},\ldots,\frac{2e+1}{t}\right)$. Since $tx=p_n\in C_{n,e,q}$ then $\widetilde{C}=\langle (p_n,0),(x,2e+1)=p_{n+1} \rangle  \in LPL^{\infty}(n+1,e,q).$ We remark that $tp_{n+1}=(p_n,0)$ (because $(2e+1)t\equiv0 \pmod{q}$), so $\widetilde{C}=\langle p_{n+1} \rangle$.
\end{proof}

In particular, if $2e+1=t^{n-1}$ we obtain the following family of cyclic perfect codes.

\begin{corollary} \label{CyclicFamily2Cor}
If $q=t^n$ where $t$ is an odd number, then the $q$-ary code $C=\langle (1,t,t^2,\ldots,t^{n-1}) \rangle  \in LPL^{\infty}(n,e,q)$, for the packing radius $e=(t^{n-1}-1)/2$.
\end{corollary}

\begin{proposition} \label{CyclicCharacterizationProp}
Let $q=(2e+1)t$. There exists a cyclic code in $LPL^{\infty}(n,e,q)$ if and only if $t^{n-1}\mid 2e+1$.
\end{proposition}

\begin{proof}
If $C \in LPL^{\infty}(n,e,q)$ is cyclic, there exists $c \in C$ with order $t^n=|C|$. Since $qc=0$ we have $t^n \mid q$, and so $t^{n-1}\mid 2e+1$. The converse follows from Corollary \ref{CyclicFamilyCor}.
\end{proof}

The next construction generalize horizontal and vertical constructions for two-dimensional perfect code in the maximum metric presented in the previous section.

\begin{proposition}[Non linear construction]
Let $C \in PL^{\infty}(n,e,q)$ and $h:C\rightarrow \mathbb{Z}_q$ be a map (called height function). If $\widehat{C}=\{(c,h(c)+(2e+1)k): c \in C, k \in \mathbb{Z}\}$, then $\widehat{C}\in PL^{\infty}(n+1,e,q)$.
\end{proposition}

\begin{proof}
Since $(2e+1)t=q$ we have $\#\widehat{C}=\#C \cdot t = t^{n+1}$, thus it suffices to prove that the minimum distance of $\widehat{C}$ is at least $2e+1$. Let $\widehat{c}_i=(c_i,h(c_i)+(2e+1)k_i)\in \widehat{C}$ with $c_i\in C$ for $i=1,2$ and suppose that $\parallel \widehat{c_1}-\widehat{c_2}\parallel_{\infty}<2e+1$. The relation $$\parallel \widehat{c_1}-\widehat{c_2}\parallel_{\infty} = \parallel c_1-c_2\parallel_{\infty}+\parallel (h(c_1)-h(c_2))+(2e+1)(k_1-k_2)\parallel_{\infty}$$ implies $\parallel c_1-c_2\parallel_{\infty}<2e$ and $\parallel (h(c_1)-h(c_2))+(2e+1)(k_1-k_2)\parallel_{\infty}<2e+1$ and so $c_1=c_2$ (because the minimum distance of $C$ is $2e+1$) and $k_1=k_2$. Therefore the minimum distance of $\widehat{C}$ is also $2e+1$ and $\widehat{C}\in PL^{\infty}(n+1,e,q)$.
\end{proof}


\begin{remark}
The non linear construction generalize horizontal and vertical constructions. Indeed, let $NL(C,h)$ be the code obtained from the non-linear construction from the code $C$ and the height function $h$. Considering $C_{a}=a+(2e+1)\in PL^{\infty}(1,e,q)$ and $h_{a}(k)=h(a+(2e+1)k)$ then $C_2(a,h_a)=NL(C_{a},h)$ and $C_1(a,h_a)=\sigma NL(C_{a},h)$ where $\sigma=(1\ 2)$.
\end{remark}

\begin{remark}
If $C$ is linear, it is possible to choose a height function in such a way that $\widetilde{C}$ is also linear, but for arbitrary choice of $h$ this is not true.
\end{remark}

\begin{remark}
Every code constructed from the non linear construction is standard. Consequently, there are codes that cannot be constructed from the non-linear construction (for example the code given in the Remark \ref{NonStandardExampleRem}). On the other hand, by Corollary \ref{MinkoskiConjStandardCor} we can obtain every linear perfect code using this construction (with good choices for the height functions) in a finite number of steps.
\end{remark}

The next construction allows to obtain perfect codes in lower dimension from a given perfect code via cartesian sections. This construction plays a fundamental role in the next section, when we introduce the concept of ordered code. 

\begin{definition}
Let $S\subseteq \Zqn$. A perfect code in $S$ is a subset $C\subseteq S$ for which there exists $e\in\mathbb{N}$ such that $S=\biguplus_{c\in C}\left( B(c,e)\cap S \right)$. In this case, $e$ is determined by $C$ (by the packing sphere condition) and we call it the packing radius of $C$.  
\end{definition}

\begin{notation}
Let $[n]=\{1,2,\ldots,n\}$. If $I\subseteq [n]$, we denote by $H_{I}=\{x\in \mathbb{Z}_{q}^{n}: x_i=0, \forall i \in I\}$ (these sets are called cartesian subgroups). We define its dimension as $\dim(H_{I})=n-\#I$.
\end{notation}

\begin{definition}
Let $I\subseteq [n]$. The orthogonal projection over $H_I$ is the unique morphism $\pi_I: \Zqn \rightarrow \Zqn$ verifying $\pi_I(e_i)=\left\{ \begin{array}{ll}
0 & \forall i \in I \\ e_i & \forall i \in I^{c}
\end{array} \right.$.
\end{definition}

\begin{notation}[Generalized balls] If $H \subseteq \mathbb{Z}_{q}^{n}$ and $e \in \mathbb{N}$, we denote by $$B(H,e)=\{x \in \Zqn: d(x,h)\leq e \textrm{ for some } h\in H\}= \bigcup_{h\in H}B(h,e).$$
\end{notation}

The following lemmas can be obtained easily from the previous definition.

\begin{lemma}\label{projLemma}
Let $I\subseteq [n]$. For $h\in H_{I}$ and $x \in \Zqn$ we have $d(x,h)\geq d(\pi_{I}(x),h)$.
\end{lemma}

\begin{lemma}\label{ExtBallLemma}
If $x \in \Zqn$, then $x\in B\left(H_{I},e\right)\leftrightarrow |x_i|_{1}\leq e$ $\forall i \in I$.
\end{lemma}


\begin{remark}
If $C \in LPL^{\infty}(n,e,q)$ and $f_C: \Zqn \rightarrow C$ is the associated error correcting function, then $B(H,e)\cap C = f_{C}(H)$.
\end{remark}

\begin{definition}
Let $C\subseteq \Zqn$ and $I\subseteq [n]$. The cartesian section of $C$ (respect to the cartesian subgroup $H_I$) is given by $C\langle I \rangle := \pi_{I}(B(H_{I},e)\cap C)=\pi_{I}\circ f_{C}(H_I)$.
\end{definition}

\begin{proposition}[Section construction] \label{SectionConstrProp}
If $C \in PL^{\infty}(n,e,q)$ and $I\subseteq [n]$, then $C\langle I \rangle$ is a perfect code in $H_{I}$ with packing radius $e$.
\end{proposition}

\begin{proof}
Let $h \in H_{I}$ and $c=\pi_{I}(f_{C}(h)) \in C\langle I \rangle$. By Lemma \ref{projLemma} $d(c,h)\leq d(f_{C}(h),h)\leq e$ and we have that $h\in B(c,e)$, so $H \subseteq \bigcup_{c\in C\langle I \rangle}B(c,e)$. Since $H=\bigcup_{c\in C\langle I \rangle}\left( B(c,e)\cap H\right)$ the covering radius of $C\langle I \rangle$ is at most $e$ (as code in $H_I$). On the other hand, if $\widehat{c}_1, \widehat{c}_2 \in C\langle I \rangle$ verify $d(\widehat{c}_1,\widehat{c}_2)\leq 2e$ and we express $\widehat{c}_{i}=\pi_{I}(c_i)$ for $i=1,2$ where $c_i \in B(H,e)\cap C$, we have \begin{equation}\label{Eq2eP1}
|c_{1}(i)-c_{2}(i)|_{1} = |\widehat{c}_1(i)-\widehat{c}_2(i)|_{1} \leq 2e, \ \forall i \in I^{c},
\end{equation}
and by Lemma \ref{ExtBallLemma} we have
\begin{equation}\label{Eq2eP2}
|c_{1}(i)-c_{2}(i)|_{1} \leq |c_{1}(i)|_{1}+|c_{2}(i)|_{1}\leq 2e \ \forall i \in I.
\end{equation}
Equations (\ref{Eq2eP1}) and (\ref{Eq2eP2}) imply $d(c_1,c_2)\leq 2e$ and since $C$ has packing radius $e$ we have $c_1=c_2$ so $\widehat{c}_1=\widehat{c}_{2}$. Therefore $C\langle I \rangle$ has minimum distance $d\geq 2e+1$ and its packing radius is at least $e$. We conclude that $C\langle I \rangle$ is a perfect code in $H_I$ with packing radius $e$.
\end{proof}

In the linear case, under some conditions we can prove that the resulting code is also linear.

\begin{proposition}\label{LinearSectionProp}
If $C \in LPL^{\infty}(n,e,q)$ is of type $i$ for all $i\in I$, then $C\langle I \rangle$ is a linear perfect code in $H_{I}$ with packing radius $e$. Moreover, $C\langle I \rangle=\pi_{I}(C)$.
\end{proposition}

\begin{proof}
We just need to check linearity and for this it suffices to prove that $\pi_{I}(C)=C\langle I \rangle$. It is clear that $C\langle I \rangle= \pi_I(C \cap B(H_I,e)) \subseteq \pi_I(C)$. For the other inclusion, let $c \in C$ and for each $i \in I$ we consider $k_i \in \mathbb{Z}$ such that $|c+(2e+1)k_i|_1 \leq e$. Since $C$ is of type $i$ for all $i\in I$, then $(2e+1)k_ie_i \in C$ and also the vector $v = \sum_{i\in I}(2e+1)k_ie_i \in C$. By Lemma \ref{ExtBallLemma} $c+v \in C \cap B(H_I,e)$ and $(c+v)_j=c_j$ for all $j\in I^{c}$, therefore $\pi_{I}(c)=\pi_{I}(c+v)\in \pi_{I}(B(H_I,e)\cap C)= C \langle I \rangle$ and we have $\pi_I(C)\subseteq C \langle I \rangle$.
\end{proof}


\section{Perfect codes in arbitrary dimensions}

\subsection{Permutation associated with perfect codes}

The type of a code is an important concept when we deal with two-dimensional perfect codes, in part because every two-dimensional linear perfect code is isometric to a code of type $2$ which have an upper triangular generator matrix (Theorem \ref{GeneratorForLPL2D}). This last property is false for greater dimensions so we need a more general concept in order to describe all perfect codes with given parameters $(n,e,q)$.

\begin{notation}
For $C \in LPL^{\infty}(n,e,q)$ we denote by $$\tau(C)=\max\{i : 1\leq i \leq n, C \textrm{ is of type }i\}.$$
\end{notation}

\begin{definition}\label{PermutationAssocDef}
Let $C \in LPL^{\infty}(n,e,q)$ with $q=(2e+1)t$ and $t>1$. We consider the following sequence:
$$\left\{ \begin{array}{ll} \wp_1= \tau(C),\ J_1=\{\wp_1\},\ C_1=C\langle J_1 \rangle & \\ 
\wp_{i+1}=\tau(C_i),\ J_{i+1}=J_{i}\cup \{\wp_{i+1}\},\ C_{i+1}=C\langle J_{i+1} \rangle & \textrm{for $1\leq i <n$}\end{array}  \right.$$
The permutation of $[n]$ associated with $C$ is $\wp(C)=(\wp_1,\wp_2,\ldots, \wp_{n})$.
\end{definition}

\begin{remark}
Since we start from a linear code C, Minkowski-Hajos Theorem (Theorem \ref{MinkowskiConjTh}) and Proposition \ref{LinearSectionProp} guarantee the existence of $\wp_i$ in each step and the linearity of the corresponding code $C_i$. Since $(2e+1)e_k \not\in H_I$ for $k\in I$, we have that the numbers $\wp_i$ are pairwise different, so $\wp \in S_n$. 
\end{remark}

\begin{example}\label{PartitionExample}
We consider the code $C=\mbox{span} \left( \begin{matrix}
1 & 3 & 0 & 0 \\ 0 & 0 & 1 & 3 \\ 3 & 0 & 1 & 0 \\ 0 & 0 & 3 & 0
\end{matrix} \right)$ over $\mathbb{Z}_{81}^4$. This code is perfect with parameters $(n,e,q)=(4,1,81)$. Let us calculate its associated permutation. In the first step we have: $$\bullet \quad \wp_1=\tau(C)=3, J_1=\{3\}, C_1=C\langle 3 \rangle= \mbox{span} \left( \begin{matrix}
1 & 3 & 0 & 0 \\ 0 & 0 & 0 & 3 \\ 3 & 0 & 0 & 0 \\ 0 & 0 & 0 & 0
\end{matrix} \right),$$ in the second step we have:
$$\bullet \quad \wp_2=\tau(C_1)=4, J_2=\{3,4\}, C_2=C\langle 3,4 \rangle= \mbox{span} \left( \begin{matrix}
1 & 3 & 0 & 0 \\ 0 & 0 & 0 & 0 \\ 3 & 0 & 0 & 0 \\ 0 & 0 & 0 & 0
\end{matrix} \right),$$  in the third step we have:
$$\bullet \quad \wp_3=\mathcal{I}(C_2)=1, J_2=\{1,3,4\}, C_3=C\langle 1,3,4 \rangle= \mbox{span} \left( \begin{matrix}
0 & 3 & 0 & 0 \\ 0 & 0 & 0 & 0 \\ 0 & 0 & 0 & 0 \\ 0 & 0 & 0 & 0
\end{matrix} \right),$$ and in the last step we have $\wp_4=\tau(C_3)=2$
so, the permutation associated with $C$ is $\wp(C)=(3,4,1,2)$.
\end{example}

\begin{definition}
We say that a perfect code $C \in LPL^{\infty}(n,e,q)$ is ordered if its associated partition is given by $\wp(C)=(n,n-1,\ldots,2,1)$.
\end{definition}

\begin{proposition}\label{TransCor}
For all $C \in LPL^{\infty}(n,e,q)$ there exists $\theta=\theta_{C} \in S_n$ such that the code $\theta(C):=\left\{\left(c_{\theta^{-1}(1)},\ldots,c_{\theta^{-1}(n)}\right): (c_1,\ldots,c_n)\in C\right\}$ is ordered. 
\end{proposition}

\begin{proof}
Let $C \in LPL^{\infty}(n,e,q)$ with $\wp(C)=(\wp_1,\wp_2,\ldots, \wp_n)$ and $\theta$ be the permutation given by $\theta(\wp_i)=n+1-i$. For all $i$ with $1\leq i < n$ we have $(2e+1)e_{\wp_{i+1}} \in C \langle \wp_1,\ldots, \wp_{i} \rangle = \pi_{H_{I}}(C)$ where $I=\{\wp_1,\ldots,\wp_i\}$, so there exists $c \in C \cap \pi_{H_{I}}^{-1}((2e+1)e_{\wp_{i+1}})$. We have $c_{\wp_{i+1}}=2e+1$ and $c_{k}=0$ for $k \not\in \{\wp_1,\wp_2,\ldots,\wp_{i+1}\}$. Since $\theta(c)_{i}=c_{\theta^{-1}(i)}$ we have $\theta(c)_{n-i}=2e+1$ and $\theta(c)_{k}=0$ for $k: n\geq k \geq n-i$, so $$
(2e+1)e_{n-i}=\pi_{H_{\theta(I)}}\left( \theta(c)  \right) \in \theta(C)\langle n, n-1, \ldots, n+1-i \rangle$$ for $1\leq i <n $. This last condition together with the fact that $(2e+1)e_{n}=\theta \left((2e+1)e_{\wp_{1}} \right) \in \theta(C)$ imply $\wp\left( \theta(C) \right)=(n,n-1,\ldots,1)$.
\end{proof}

\begin{notation}
We denote by $LPL^{\infty}(n,e,q)_{o}= \{ C \in LPL^{\infty}(n,e,q): C \textrm{ is ordered} \}$. 
\end{notation}

\begin{example}
Let $C$ be the code defined in Example \ref{PartitionExample}. The permutation $\theta=(1\ 2)(3\ 4)$ verify $\theta(\wp_i)=5-i$, so the resulting code $\theta(C)\in LPL^{\infty}(4,1,81)_{o}$. We remark that this permutation is not unique, for example if we take $\tau=(1\ 3\ 4\ 2)$ we have $\tau(C) \in LPL^{\infty}(4,1,81)_{o}$ and $\tau(C)\neq \theta(C)$.
\end{example}

\subsection{Perfect matrices}
 
In this section we characterize matrices associated with perfect codes in the maximum metric.

\begin{definition}
Let $q=(2e+1)t$. A matrix $M \in \nabla_{n}(2e+1)$ is a $(e,q)$-perfect matrix if there exist matrices $A \in \nabla_n(t)$ and $B \in \nabla_{n}(1)$ such that $AM=qB$.
\end{definition}

\begin{remark}
For $n=2$ a matrix $M=\left( \begin{matrix} 2e+1 & a \\ 0 & 2e+1
\end{matrix}\right)$ is $(e,q)$-perfect if only if there exists $x,y \in \mathbb{Z}$ satisfying 
$ \left( \begin{matrix} t & x \\ 0 & t
\end{matrix}\right) \left( \begin{matrix} 2e+1 & a \\ 0 & 2e+1
\end{matrix}\right)= \left( \begin{matrix} q & qz \\ 0 & q
\end{matrix}\right)$ and this is equivalent to $ta+(2e+1)x=qz$ or $qz-(2e+1)x=ta$. This last diophantine equation has solution if and only if $\gcd(q,2e+1)=2e+1 \mid ta$ which is equivalent to $a=kh_1$ for some $k\in \mathbb{Z}$ (where $h_1=\frac{2e+1}{\gcd(2e+1,t)}$). Is summary, a $2\times 2$ matrix is a $(e,q)$-perfect matrix if and only if it is the generating matrix of a type $2$ perfect code in $LPL^{\infty}(2,e,q)$. 
\end{remark}

\begin{proposition}\label{PerfCondProp}
If $q=(2e+1)t$ and $M$ is a $n\times n$ integer matrix with rows $M_1,M_2,\ldots,M_{n}$, then $M$ is a $(e,q)$-perfect matrix if and only if the following condition are satisfied:
\begin{enumerate}
\item $M$ is upper triangular,
\item $M_{ii}=2e+1$,
\item $t\overline{M_i} \in \textrm{span}(\overline{M_{i+1}},\ldots,\overline{M_{n}})$ for $1\leq i < n$,
\end{enumerate} 
where for $x\in\mathbb{Z}^{n}$ we denote by $\overline{x}=x+q\mathbb{Z}^{n} \in \mathbb{Z}_{q}^{n}$ the residual class of $x$ modulo $q$.
\end{proposition}

\begin{proof}
Conditions (1) and (2) are equivalent to $M \in \nabla_n(2e+1)$ and condition (3) is equivalent to the existence of integers $\alpha_{ij} \in \mathbb{Z}$ and vectors $B_i \in \mathbb{Z}^{n}$ for $1\leq i <j\leq n$ verifying $ tM_i = \sum_{j=i+1}^n \alpha_{ij}M_j+qN_i$. These equations can be expressed in matricial form as $AM=qB$ where the matrix $A$ is upper triangular with $A_{ij}=\left\{ \begin{array}{ll}
t & \textrm{for } i=j \\ -\alpha_{ij} & \textrm{for }i<j  
\end{array} \right.$ and $B$ has rows $B_1,B_2,\ldots,B_n$.
\end{proof}

\begin{lemma}\label{RelPerfLemma}
Let $q=(2e+1)t$, $C \in LPL^{\infty}(n,e,q)$ and $H$ be a $k$-dimensional cartesian subgroup of $\Zqn$. If $S\subseteq C \cap H$ and $\#S=t^{k}$, then $S=C \cap H$ and the code $S$ is a perfect code in $H$ with packing radius $e$.
\end{lemma}

\begin{proof}
We have that $t^{k}=\#S \leq \# (C \cap H) \leq \frac{q^k}{(2e+1)^k}=t^k$ (the last inequality is consequence of the sphere packing condition) so $S=C\cap H$. Let $e'$ be the packing radius of $S$. Since $C$ has packing radius $e$ we have $e'\geq e$. By the sphere packing condition $(2e'+1)^k \leq \frac{q^k}{\# S}=(2e+1)^k$ hence $e'=e$ and $(2e+1)^k \cdot \# S = q^k$, therefore $S=C\cap H$ is a perfect code in $H$ with packing radius $e$. 
\end{proof}

\begin{proposition}\label{PerfectGMProp}
Every ordered perfect code $C \in LPL^{\infty}(n,e,q)$ has a generator matrix which is a $(e,q)$-perfect matrix.
\end{proposition}

\begin{proof}
Let $C \in LPL^{\infty}(n,e,q)$. By the Hermite normal form theorem we have a generator matrix $M$ for $C$ which is upper triangular. Let $M_1,M_2,\ldots,M_n$ be the rows of $M$ and we denote by $m_i=M_{ii}$ the elements in the principal diagonal. Multiplying by $-1$ if it were necessary we can suppose that each $m_i>0$ ($m_i\neq 0$ because $M$ is non-singular). We will prove the following assertion by induction:
\begin{equation}\label{EqInduction}
\left\{ \begin{array}{l}
 t \overline{M_{n-i}} \in \mbox{span}\left( \overline{M_{n-(i-1)}},\overline{M_{n-(i-2)}},\ldots, \overline{M_{n}} \right) \\ m_{n-(i-1)}=m_{n-(i-2)}=\ldots=m_{n}=2e+1 \end{array} \right.
\end{equation} for $1\leq i < n$, where as usual $\overline{X}=X+q\mathbb{Z}^n \in \mathbb{Z}_q^n$ is the residual class modulo $q$ of $X \in \mathbb{Z}^n$. For $i=1$ we express $m_n=(2e+1)a+r$ with $a$ and $b$ non-negative integer $0\leq r < 2e+1$. Since $C$ is of type $n$ (because $C$ is ordered) we have that $v=(2e+1)e_n \in \Lambda_{C}$ and also $r e_n = M_n - av \in \Lambda_C$. The packing radius of $\Lambda_C$ (which is equal to the packing radius of $C$) is $e$ and consequently its minimum distance is $2e+1$, but $|| re_n||_{\infty}=r$ which imply $r=0$ and $M_n=av$. Substituting $M_n$ by $v$ we have another generator matrix $M'$ for $C$, since $\det(M)=\det(M')=\det(\Lambda_{C})$ and $\det(M)=a\det(M')$ we have $a=1$ and $m_n=2e+1$. Since $C$ is ordered, the code $C\langle n \rangle$ is of type $n-1$ and since $m_{n-1}e_{n-1} \in C\langle n \rangle$, using a similar argument as in the proof of $m_n=2e+1$ we can prove that $m_{n-1}=2e+1$, so $t \overline{M_{n-1}}\in H_{\{1,\ldots,n-1\}}\cap C$. On the other hand, since $M_n= (2e+1)e_n$ we have $\mbox{span}(\overline{M_n})\subseteq C \cap H_{\{1,\ldots,n-1\}}$ and $\#\mbox{span}(\overline{M_n})=t$, thus by Lemma \ref{RelPerfLemma} we have $H_{\{1,\ldots,n-1\}}\cap C=\mbox{span}(\overline{M_n})$ so the assertion (\ref{EqInduction}) is true for $i=1$. Now consider $j$ with $2\leq  j< n$ and let us suppose that the assertion (\ref{EqInduction}) is true for $i$ with $1\leq i<j$. By inductive hypotesis $t \overline{M_{n-i}} \in \mbox{span}\left( \overline{M_{n-(i-1)}},\overline{M_{n-(i-2)}},\ldots, \overline{M_{n}} \right)$ for $1\leq i <j$ so linear construction (Prop. \ref{LinearConstrProp}) guarantees that $C'= \mbox{span}\left( \overline{M_{n-(j-1)}},\overline{M_{n-(j-2)}},\ldots, \overline{M_{n}} \right)$ is a perfect code in $H_{\{1,2,\ldots,n-j\}}$ with packing radius $e$. In particular $\# C'=t^{j}$ and by Lemma \ref{RelPerfLemma} we have that $C' = C \cap H_{1,2,\ldots,n-j}$. Since $C$ is ordered, $C\langle n-(j-1),\ldots,n \rangle$ is of type $n-j$ and using that $m_{n-j}e_{n-j} \in C\langle n-(j-1),\ldots,n \rangle$ and a similar argument used in the proof of $m_n=2e+1$ we can prove that $m_{n-j}=2e+1$, then $t \overline{M_{n-j}}\in H_{\{1,2,\ldots,n-j\}}\cap C =\mbox{span}\left( \overline{M_{n-(j-1)}},\overline{M_{n-(j-2)}},\ldots, \overline{M_{n}} \right)$, so assertion (\ref{EqInduction}) is true for $i=j$. Finally, assertion ($\ref{EqInduction}$) for $1\leq i <n$ and Proposition \ref{PerfCondProp} imply that $M$ is an $(e,q)$-perfect matrix.
\end{proof}

\begin{remark}\label{PerfectHermiteRemark}
In order to obtain a $(e,q)$-perfect generator matrix for an ordered perfect code $C\in LPL^{\infty}(n,e,q)$ from a given generator matrix we can apply the Hermite normal form algorithm and multiply some rows by $-1$, if necessary.
\end{remark}

\begin{definition}
We say that a matrix $M \in \nabla_n(2e+1)$ is reduced if $|M_{ij}|\leq e$ for $1\leq i< j \leq n$.
\end{definition}

\begin{notation}
We denote by $\mathcal{P}_{n}(e,q)=\{M \in \nabla_n(2e+1): M \textrm{ is }(e,q)-\textrm{perfect}\}$. The subset of reduced matrices in $\nabla_{n}(2e+1)$ and $\mathcal{P}_{n}(e,q)$ is denoted by $\nabla_{n}(2e+1)_{\textrm{red}}$ and $\mathcal{P}_{n}(e,q)_{\textrm{red}}$ respectively.
\end{notation}

\begin{proposition}
Let $M,M' \in \mathcal{P}_{n}(e,q)$. If $M_{ij}\equiv M'_{ij} \pmod{2e+1}$ then $\mbox{span}(M)=\mbox{span}(M')$.
\end{proposition}

\begin{proof}
We observe that a reduced $(e,q)$-perfect generator matrix for a code $C \in LPL^{\infty}(n,e,q)$ is just a modified version of the Hermite normal form, so $\mbox{span}(M)=\mbox{span}(M')$ is a consequence of the uniqueness of the Hermite normal form.
\end{proof}

\begin{corollary}\label{SurjectionCor}
There is a surjection $\mathcal{P}_{n}(e,q)_{\textrm{red}} \twoheadrightarrow LPL^{\infty}(n,e,q)_{o}$ given by $M \mapsto \mbox{span}(M)/q\mathbb{Z}^{n}$.
\end{corollary}

\begin{proof}
By Proposition \ref{PerfectGMProp} and Remark \ref{PerfectHermiteRemark} we can obtain a $(e,q)$-perfect generator matrix $M$ from the Hermite normal form of any generator matrix with the condition $0\leq M_{ij}<2e+1$ if $i<j$. For $i=2,3,\ldots,n$ and for $1\leq j <i$, if the element $ji$ is greater than $e$ we can substract the $i$-th row to the $j$-th row obtaining a new equivalent matrix whose element $ji$ has absolute value at most $e$. Repeating this process we obtain a reduced $(e,q)$-perfect generator matrix for a given ordered code $C \in LPL^{\infty}(n,e,q)_{o}$.
\end{proof}

\begin{corollary}\label{MaxInequalityCor}
Let $q=(2e+1)t$. We have the following inequality:
\begin{equation}\label{MaximalInequality}
\log_{2e+1}\left(\# LPL^{\infty}(n,e,q)\right) \leq \binom{n}{2}
\end{equation}
\end{corollary}

\begin{proof}
Using Lemma \ref{SurjectionCor} we obtain: $$\# LPL^{\infty}(n,e,q)\leq \# \mathcal{P}_{n}(e,q)_{\mbox{red}} \leq \# \nabla_{n}(2e+1)_{\mbox{red}}=(2e+1)^{\binom{n}{2}}.$$
\end{proof}

\begin{definition}\label{MaximalDef}
If the parameters $(n,e,q)$ verify equality in Corollary \ref{MaxInequalityCor} we say that the parameters $(e,q)$ are $n$-maximals and we call a code with these parameters a maximal code.
\end{definition}

\begin{remark}
If $(e,q)$ is $n$-maximal then $\mathcal{P}_{n}(e,q)_{\mbox{red}} = \nabla_{n}(2e+1)_{\mbox{red}}$.
\end{remark}

\subsection{The $n$-maximal case}

In this subsection we show that there are infinitely many maximal codes in each dimension establishing conditions which guarantee maximality. We extend some results obtained for two-dimensional codes in Section 3, to maximal codes including a parametrization theorem for such codes and for their isometry and isomorphism classes.

\begin{lemma}\label{iMaximalLemma}
If $(e,q)$ is $n$-maximal, then $(e,q)$ is $i$-maximal for all $i, 1\leq i \leq n$.
\end{lemma}

\begin{proof}
Let $(e,q)$ be an $n$-maximal pair and $M'\in \nabla_{i}(2e+1)$ with $1\leq i <n$. We consider the matrix $M= \left( \begin{matrix}
(2e+1)I_{n-i} & 0 \\ 0 & M'
\end{matrix}  \right) \in \nabla_{n}(2e+1)$, since $(e,q)$ is $n$-maximal there exists $A \in \nabla_{n}(t), B \in \nabla_{n}(1)$ such that $AM=qB$. If we denote by $A'$ and $B'$ the submatrices consisting of the last $i$ rows and the last $i$ columns of $A$ and $B$ respectively. Clearly, $A' \in \nabla_{i}(t)$ and $B' \in \nabla_i(1)$ and $A'M'=qB'$, therefore $M'$ is $(e,q)$-perfect, so the pair $(e,q)$ is $i$-maximal.
\end{proof}

\begin{lemma}\label{MaximalEquivLemma}
Let $q=(2e+1)t$ and $\overline{X}=X+q\mathbb{Z}^{n}\in \mathbb{Z}_{q}^n$ be the residual class of $X\in \mathbb{Z}^n$ modulo $q$. The following assertions are equivalent:
\begin{itemize}
\item[(i)] $(e,q)$ is $n$-maximal.
\item[(ii)] For all $M \in \nabla_{n}(2e+1)$, there exists $A \in \nabla_{n}(t), B \in \nabla_{n}(1)$ such that $AM=qB$.
\item[(iii)] $t \mathbb{Z}_q^{i} \subseteq \mbox{span}(\overline{M})$ for all $M \in \nabla_{i}(2e+1)$ and for all $i$, $1\leq i <n$.
\end{itemize}
\end{lemma}

\begin{proof}
We have (ii) $\Leftrightarrow \nabla_{n}(2e+1)=\mathcal{P}_{n}(e,q) \Leftrightarrow (i)$. Note that if (iii) holds then condition (3) in Proposition \ref{PerfCondProp} is always satisfied, thus (iii) $\Rightarrow$ (ii). In order to prove (ii) $\Rightarrow$ (iii), by Lemma \ref{iMaximalLemma} it suffices to prove (ii) $\Rightarrow t \mathbb{Z}_{q}^{n-1} \subseteq \mbox{span}(\overline{M})$ for all $M \in \nabla_{n-1}(2e+1)$. Let $M' \in \nabla_{n-1}(2e+1)$ and $w \in \mathbb{Z}^{n-1}$, we want to prove that $t \overline{w} \in \mbox{span}(\overline{M})$. Consider the matrix $M= \left( \begin{matrix}
2e+1 & w \\ 0^{t} & M' \end{matrix}  \right) \in \nabla_{n}(2e+1)$. By (ii) there exist matrices $A\in \nabla_{n}(t), B \in \nabla_n(1)$ such that $AM=B$. Expressing $A= \left(\begin{matrix}
t & v \\ 0^t & A' \end{matrix} \right)$ and $B= \left(\begin{matrix}
1 & u \\ 0^t & B' \end{matrix} \right)$ with $A' \in \nabla_{n-1}(t)$ and $B' \in \nabla_{n-1}(1)$, from the equality $AM=qB$ we obtain $tw+vM'=qu$, thus $t\overline{w}=-v \overline{M'}\in \mbox{span}(\overline{M})$.
\end{proof}

\begin{lemma}
If $(2e+1)^n \mid q$, then $(2e+1)^{n-1} \Zq^{n-1}\subseteq \mbox{span}(\overline{M})$ for all $M \in \nabla_{n-1}(2e+1)$.
\end{lemma}

\begin{proof}
For $n=1$ the assertion is true because $(0)\subseteq \mbox{span}(\overline{M})$ and for $n=2$ the assertion is true because $(2e+1)\mathbb{Z}_q \subseteq \mbox{span}(2e+1)=(2e+1)\mathbb{Z}_q$. Let us suppose that the assertion is true for $n-1$ where $n\geq 3$ and $(2e+1)^n \mid q$. Let $M \in \nabla_{n-1}(2e+1)$ written as $M=\left( \begin{matrix} 2e+1 & w \\ 0^t & M' \end{matrix} \right)$ with $M' \in \nabla_{n-2}(2e+1)$ and $w \in \mathbb{Z}^{n-2}$. Since $(2e+1)^{n-2}\mid q$ we have $(2e+1)^{n-2}\mathbb{Z}_q^{n-2} \subseteq \mbox{span}(\overline{M})$, then  $(2e+1)^{n-2}H_{\{1\}}\subseteq \mbox{span}\overline{(0^t, M')}$ and we obtain the following chain of inequalities: $$(2e+1)^{n-1}H_{\{1\}}\subseteq (2e+1)^{n-2}H_{\{1\}}\subseteq \mbox{span}\overline{(0^t, M')} \subseteq \mbox{span}(\overline{M}).$$ To conclude the proof we need to show that $(2e+1)^{n-1}\overline{e_1} \in \mbox{span}(\overline{M})$. We have that $ (2e+1)^{n-1}\overline{e_1} - (2e+1)^{n-2} (2e+1,w)=(0,-(2e+1)^{n-2}w) \in (2e+1)^{n-2}H_{\{1\}} \subseteq \mbox{span}(\overline{M}), $ so $(2e+1)^{n-1}\overline{e_1} \in \mbox{span}(\overline{M})$. In conclusion, we have that $$(2e+1)^{n-1}\mathbb{Z}_q^{n-1}=(2e+1)^{n-1}\mathbb{Z}\overline{e_1}\oplus (2e+1)^{n-1}H_{\{1\}}\subseteq \mbox{span}(\overline{M}).$$
\end{proof}

\begin{corollary}\label{PerfectCor}
If $(2e+1)^n \mid q$ then $(2e+1)^{i} \Zq^{i}\subseteq \mbox{span}(\overline{M})$ for all $M \in \nabla_{i}(2e+1)$ and for all $i$, $1\leq i <n$.
\end{corollary}

\begin{theorem}
Let $q=(2e+1)t$. The pair $(e,q)$ is $n$-maximal if and only if $(2e+1)^{n-1} \mid t$.
\end{theorem}  

\begin{proof}
First, we suppose that $(2e+1)^{n-1}\mid t$ (or equivalently $(2e+1)^{n}\mid q$). By Corollary \ref{PerfectCor}, for all $M\in \nabla(2e+1)$ and $1\leq i <n$ we have:
$$t \mathbb{Z}_q^{i}\subseteq (2e+1)^{n-1}\mathbb{Z}_q^{i}\subseteq (2e+1)^i \mathbb{Z}_{q}^i \subseteq \mbox{span}(\overline{M}),$$ and by Lemma \ref{MaximalEquivLemma} the pair $(e,q)$ is $n$-maximal. Now we suppose that $(e,q)$ is $n$-maximal and consider the bidiagonal matrix $M \in \nabla_{n}(2e+1)$ which has $1$ in the secondary diagonal (i.e. in the diagonal above the principal diagonal). Since $(e,q)$ is $n$-maximal, there exists $A \in \nabla_{n}(t), B \in \nabla_{n}(1)$ such that $AM=qB$. If we denote the first row of $A$ by $A_1=(a_{11},a_{12},\ldots, a_{1n})$ and the first row of $B$ by $B_1$, we have that $qB_1=(q,a_{11}+(2e+1)a_{12},a_{12}+(2e+1)a_{13},\ldots, a_{1,n-1}+(2e+1)a_{1n})$ using $a_{11}=t$ we deduce that: $$t+(-1)^{n}(2e+1)^{n-1}a_{1n} = \sum_{i=1}^{n-1} (-2e-1)^{i-1}(a_{1i}+(2e+1)a_{1,i+1})\equiv 0 \pmod{q}.$$ If $h\in \mathbb{Z}$ is such that $t+(-1)^n (2e+1)^{n-1}a_{1n}=qh$ we have $t(1-(2e+1)h)=(-1)^{n+1}(2e+1)^{n-1}a_{1n}$, since $\gcd(1-(2e+1)h,2e+1)=1$ we have $(2e+1)^{n-1}\mid t$.
\end{proof}

Since $\nabla_{n}(2e+1)=\mathcal{P}_n(e,q)$ holds for maximal codes, in this case we obtain the following parametrization for ordered codes which generalize the first part of Theorem \ref{parametrizationTheorem}.

\begin{theorem}\label{GralParamTheoremP1}
Let $(e,q)$ be an $n$-maximal pair. There is a parametrization $$\psi: \nabla_{n}(2e+1)_{\mbox{red}} \rightarrow LPL^{\infty}(n,e,q)_{o}$$ given by $\psi(M)=\mbox{span}(M)/q\mathbb{Z}^n$.
\end{theorem}


Next we study isomorphism classes of perfect codes.

\begin{notation}
An unimodular integer matrix is a square matrix with determinant $1$ or $-1$. We denote by $\Gamma_n=\{M \in M_{n}(\mathbb{Z}): M \textrm{ is unimodular} \}$. If $A, B \in M_n(\mathbb{Z})$ we say that $A$ and $B$ are $\Gamma_n$-equivalent if there exists $U,V \in \Gamma_n$ such that $A=UBV$, we denote $A\underset{\Gamma}{\sim}B$ for this equivalence relation.
\end{notation}

We remark that two matrices $A$ and $B$ are $\Gamma$-equivalent if we can obtain one from the other through a finite number of elementary operations on the rows and on the columns. For $X \subseteq M_n(\mathbb{Z})$ we denote by $X/\Gamma_n$ the quotient space for this equivalence relation. 

\begin{theorem}\label{GralParamTheoremP3}
Let $(e,q)$ be an $n$-maximal pair. There is a parametrization $$\psi_{\mathcal{A}}: \frac{\nabla_{n}(2e+1)_{\mbox{red}}}{\Gamma_{n}} \rightarrow \frac{LPL^{\infty}(n,e,q)_{o}}{\mathcal{A}}$$ given by $\psi_{\mathcal{A}}(M)=[\psi(M)]_{\mathcal{A}}$ (where $\psi$ is as in Theorem \ref{GralParamTheoremP1} and $[C]_{\mathcal{A}}$ denotes the isomorphism class of $C$).
\end{theorem}

\begin{proof}
If $M$ is the generator matrix for a linear code $C\subseteq \mathbb{Z}_q^n$ then the matrix $qM^{-1}$ has integer coefficient and their Smith normal form determines the isomorphism class of $C$ (as abelian group). On the other hand for $M_1,M_2 \in \nabla_{n}(2e+1)_{\textrm{red}}$ we have the following equivalences:$$ \mbox{span}(\overline{M_1})\sima\mbox{span}(\overline{M_2}) \Leftrightarrow qM_1^{-1} \simG qM_2^{-1} \Leftrightarrow \exists U,V \in \Gamma_n : UqM_1^{-1}V=qM_2^{-1}$$ $$ \Leftrightarrow \exists U,V \in \Gamma_n : V^{-1}M_1 U^{-1}=M_2 \Leftrightarrow M_1 \simG M_2,$$ so $\psi_{\mathcal{A}}$ is well defined and is injective. Since $\psi$ is surjective then $\psi_{\mathcal{A}}$ is surjective, therefore $\psi_{\mathcal{A}}$ is a bijection.
\end{proof}

The next goal is to characterize what are the possible group isomorphism classes that can be represented by maximal perfect codes.

\begin{definition}\label{admissibleDef}
Let $G=\mathbb{Z}_{d_1}\times \ldots \times \mathbb{Z}_{d_m}$ with $d_1|d_2|\ldots|d_m$. We say that $G$ is an $(n,e,q)$-admissible structure if there exist $C \in LPL^{\infty}(n,e,q)$ such that $C\simeq G$ as abelian groups.
\end{definition}

\begin{lemma}\label{EquivGamma2DLemma}
For $a,x$ and $y$ non-zero integers we have $\left( \begin{matrix}
a & 0 \\ 0 & axy \end{matrix} \right)\simG  \left( \begin{matrix}
ay & a \\ 0 & ax \end{matrix} \right).$
\end{lemma}

\begin{proof} We have the following chain of $\Gamma$-equivalence:
$$\left( \begin{matrix} a & 0 \\ 0 & axy \end{matrix} \right)\simG  \left( \begin{matrix} a & a \\ 0 & axy \end{matrix} \right)\simG  \left( \begin{matrix} a & a \\ ax & axy+ax \end{matrix} \right)\simG \left( \begin{matrix} a & a-ay \\ ax & ax \end{matrix} \right)$$ $$\simG \left( \begin{matrix} a & a-ay \\ 0 & ax \end{matrix} \right)\simG  \left( \begin{matrix} ay & a \\ 0 & ax \end{matrix} \right) $$
\end{proof}

\begin{lemma}\label{TriangularFormLemma}
If $M$ is a $n\times n$ integer matrix with determinant $(2e+1)^n$, then $\Gamma_{n}M\Gamma_{n} \cap \nabla_{n}(2e+1)_{\mbox{red}} \neq \emptyset$. Moreover, $M$ is $\Gamma$-equivalent to a bidiagonal matrix $A \in \nabla_{n}(2e+1)_{\mbox{red}}$.
\end{lemma}

\begin{proof}
For $n=1$, $\det(M)=2e+1$ implies $M=(2e+1)\in \nabla_{1}(2e+1)_{\mbox{red}}$. 
Let us suppose that the result is true for $n-1$ and we consider a $n\times n$ integer matrix $M$ with $\det(M)=(2e+1)^n$. By Smith normal form $M \simG D$ where $D=   \mbox{diag}(d_1,d_2,\ldots,d_n)$ is a diagonal matrix with $d_1 \mid d_2 \mid \ldots \mid d_n$ and $d_1d_2\ldots d_n = (2e+1)^{n}$, in particular $d_n=(2e+1)x$ and $2e+1=d_1 y$ for some integers $x$ and $y$. Permuting the second and $n$th rows of $D$ and then the second and $n$th column we have $D \simG \widetilde{D}:=\mbox{diag}(d_1,d_n,d_3,\ldots, d_{n-1},d_2)$. Applying Lemma \ref{EquivGamma2DLemma} with $d_n=d_1xy$ we obtain $\widetilde{D} \simG \left( \begin{matrix} 2e+1 & v \\ 0^{t} & D_0 \end{matrix} \right)$ where $v=(d_1,0,\ldots,0)\in \mathbb{Z}^{n-1}$ and $D_0=\mbox{diag}(d_1 x, d_3, \ldots, d_{n-1},d_2)$. By inductive hypothesis there exists unimodular matrices $U_0,V_0 \in \Gamma_{n-1}$ such that $U_0 D_0 V_0 \in \nabla_{n-1}(2e+1)$ with $U_0D_0V_0$ bidiagonal, thus 
$$ \left( \begin{matrix} 1 & 0 \\ 0^{t} & U_0   \end{matrix} \right)  \left( \begin{matrix} 2e+1 & v \\ 0^{t} & U_0  \end{matrix} \right)\left( \begin{matrix} 1 & 0 \\ 0^{t} & V_0  \end{matrix} \right)=\left( \begin{matrix} 2e+1 & vV_0  \\ 0^{t} & U_0D_0V_0  \end{matrix} \right)$$ is a bidigonal matrix in  $\nabla_{n}(2e+1)$. Since this matrix have $2e+1$ in the main diagonal, we can obtain a reduced matrix from this, by applying some elementary operations on rows, thus the result holds for $n$.
\end{proof}

\begin{corollary}
If we denote by $M_n(\mathbb{Z},det=D)$ the set of matrices $M \in M_{n}(\mathbb{Z})$ with $\det(M)=D$, each equivalence class in $\nabla_{n}(2e+1)_{\textrm{red}}/\Gamma_n$ is contained in exactly one equivalence class in $M_n(\mathbb{Z},det=(2e+1)^n)/\Gamma_n$. Moreover, both quotient sets have the same number of elements.
\end{corollary}

\begin{theorem}\label{AdmissibleTheorem}
Let $(e,q)$ be an $n$-maximal pair where $q=(2e+1)t$ and $G=\mathbb{Z}_{d_1}\times \ldots \times \mathbb{Z}_{d_n}$ with $d_1 | d_2 | \ldots | d_n$. Then $G$ is a $(n,e,q)$-admissible structure if and only if $d_1d_2\ldots d_n=t^n$ and $d_n|q$.
\end{theorem}

\begin{proof}
The direct implication follows from the fact that if $C\in LPL^{\infty}(n,e,q)$ then $\#C=t^n$ and $qC=\{0\}$ (because $C \subseteq \Zqn$). We denote by $\mathcal{D}=\{(d_1,\ldots, d_n)\in \mathbb{N}: d_1|\ldots|d_n, d_1\ldots d_n=t^n, d_n|q\}$. To prove the converse implication it suffices to prove that $\#\mathcal{D}= \# LPL^{\infty}(n,e,q)_{\mbox{red}}/\mathcal{A}$ (where $X/\mathcal{A}$ denotes the set of isomorphism classes of codes in $X$). By Theorem \ref{GralParamTheoremP3}, Lemma \ref{TriangularFormLemma} and the Smith normal form theorem we have that $\# LPL^{\infty}(n,e,q)_{\mbox{red}}/\mathcal{A}=\#\mathcal{F}$ where $\mathcal{F}=\{(f_1,\ldots, f_n)\in \mathbb{N}: f_1|\ldots|f_n, f_1\ldots f_n=(2e+1)^n\}$, so it suffices to prove that $\# \mathcal{D} = \# \mathcal{F}$. We consider $X=\{(x_1,\ldots,x_n)\in \mathbb{N}^n: x_1\mid \ldots \mid x_n \mid q\}$ and the involution $\psi: X \rightarrow X$ defined by $\psi(x_1,\ldots,x_n)=(y_1,\ldots, y_n)$ where $x_iy_j=q$ if $i+j=n+1$. For $x=(x_1,\ldots,x_n)\in X$ we denote by $p(x)=x_1\ldots x_n$. Since $(2e+1)^{n-1}\mid t$, then $\mathcal{F}\subseteq X$ and the property $p(\psi(a))\cdot p(a)=q^n$ imply $\psi(\mathcal{F})=\mathcal{D}$ and $\#\mathcal{F}=\#\mathcal{D}$.
\end{proof}

The involution argument in the above proof give us the following corollary.

\begin{corollary}
Let $q=(2e+1)t$ with $(2e+1)^{n-1}\mid t$ and $C \in LPL^{\infty}(n,e,q)$ with generator matrix $M$. If the Smith normal form of $M$ is given by $D=\mbox{diag}(d_1,\ldots,d_n)$, then $C \simeq \mathbb{Z}_{{q}/{d_n}}\times \mathbb{Z}_{{q}/{d_{n-1}}}\times \ldots \times \mathbb{Z}_{{q}/{d_1}}$.
\end{corollary}

\begin{remark}
The previous corollary does not hold for non-maximal codes.
\end{remark}

The following corollary give us the number of isomorphism classes of maximal perfect codes in $LPL^{\infty}(n,e,q)$.

\begin{corollary}\label{NumberofIsomClassesCor}
Let $q=(2e+1)t$ with $(2e+1)^{n-1}\mid t$ and $f(x)$ be the generating function $f(x)=\frac{1}{(1-x)(1-x^2)\ldots(1-x^n)}$. If $\nu_p(m)$ is the exponent of the prime $p$ in the factorial decomposition of $m$, then the number of isomorphism classes of perfect codes in $LPL^{\infty}(n,e,q)$ is given by:
$$ \prod_{p\mid 2e+1} [x^{n\nu_{p}(2e+1)}]f(x).$$
In particular for $n=2$ this number is given by $$\prod_{p}[x^{2\nu_p(2e+1)}]\frac{1}{(1-x)(1-x^2)}=\prod_{p}\left(\nu_{p}(2e+1)+1\right)=\sigma_0(2e+1),$$ the number of divisor of $2e+1$ (according with Corollary \ref{NumberOfCodes2DCor}, since $\gcd(2e+1,t)=2e+1$). For $n=3$ this number is given by $$\prod_{p}[x^{3\nu_p(2e+1)}]\frac{1}{(1-x)(1-x^2)(1-x^3)}=\prod_{p} \lceil {3}/{4}\cdot (\nu_p(2e+1)+1)^2 \rfloor$$ where $\lceil x \rfloor$ denotes the nearest integer to $x$. In particular when $2e+1$ is square-free this number is $3^{\omega(2e+1)}$ where $\omega(n)$ is the number of distinct prime divisors of $n$.
\end{corollary}

\begin{proof}
Let $X(\alpha)=\{(x_1,\ldots,x_n)\in\mathbb{N}^{n}: x_1 \leq \ldots \leq x_n, x_1+\ldots+x_n = n\alpha\}$ for $\alpha\in\mathbb{Z}^{+}$ and $\nu_p(a_1,\ldots,a_n):=(\nu_p(a_1),\ldots, \nu_p(a_n))$ (where $\nu_p(m)$ denote the exponent of $p$ in $m$). If $\mathcal{F}$ is as in the proof of Theorem \ref{AdmissibleTheorem}, then for each prime divisor $p\mid 2e+1$ and for each $a\in \mathcal{F}$ we have $\nu_{p}(a)\in X(\nu_p(2e+1))$. In this way we have a bijection between $\mathcal{F}$ and $\prod_{p}X(\nu_p(2e+1))$ where $p$ runs over the prime divisors of $2e+1$, in particular the number of isomorphism classes of $(n,e,q)$-codes (with $(2e+1)^{n-1}\mid t$) is given by $\#\mathcal{F}=\prod_{p\mid 2e+1}\#X(\nu_p(2e+1))$. With the standard change of variable $x_i=y_n+\ldots+y_{n+1-i}$ for $1\leq i \leq n$ we have $\#X(\alpha)=\#\{(y_1,\ldots,y_n)\in \mathbb{N}^{n}: y_1+2y_2+\ldots+ny_n=n\alpha\}$ which clearly is the coefficient of $x^{n\alpha}$ in the generating function $f(x)=\frac{1}{(1-x)(1-x^2)\ldots(1-x^{n})}$. For $n=2$ and $n=3$ we have the well known formulas $f(x)=\sum_{n=0}^{\infty}[\frac{n+2}{2}]x^{n}$ and $f(x)=\sum_{n=0}^{\infty}\lceil \frac{(n+3)^2}{12} \rfloor x^n$ (see for example p.10 of \cite{Hardy}) respectively.
\end{proof}

\subsection{The $n$-cyclic case}

By Proposition \ref{CyclicCharacterizationProp}, the set $LPL^{\infty}(n,e,q)$ contain a cyclic code if and only if $t^{n-1}| 2e+1$. If this condition is satisfied we say that $(e,q)$ is an $n$-cyclic pair. In this case we can also obtain a characterization of the admissible structures.


\begin{theorem}\label{AdmStrCyclicCaseTh}
Let $(e,q)$ be an $n$-cyclic pair where $q=(2e+1)t$ and $d_1, d_2, \ldots ,d_n$ be  positive integers verifying $d_1|\ldots|d_n$ and $d_1\cdots d_n=t^n$ then there exists $C \in LPL^{\infty}(n,e,q)$ such that $C \simeq \Z_{d_1}\times \ldots\times \Z_{d_n}$.
\end{theorem}

\begin{proof}
By Lemma \ref{TriangularFormLemma} there exists an integer matrix $A=(a_{ij})_{1\leq i,j\leq n} \in \nabla_n(t)$ with $A \simG \mbox{diag}(d_1,\ldots,d_n)$. We define $M \in \nabla_n(2e+1,\mathbb{Q})$ recursively as following:
\begin{equation}\label{EqRecurrenceForM}
\left\{ \begin{array}{ll}
M_n= (2e+1)e_n &\\
M_{i}= (2e+1)e_i - \sum_{k=i+1}^n ({a_{ik}}/{t})  M_{k} & \textrm{for }1\leq i <n,
\end{array} \right.
\end{equation} 
where $M_i$ denote the $i$th row of $M$. Using $t^{n-1}\mid 2e+1$, it is not difficult to prove by induction that $M_{i}\in t^{i-1}\Z^{n}$ for $1\leq i \leq n$, which implies that the matrix $M$ has integer coefficient, hence $M \in \nabla_n(2e+1)$. Equation (\ref{EqRecurrenceForM}) can be written in matricial form as $AM=qI$, in particular $M \in \mathcal{P}_{n}(e,q)$ (that is, $M$ is $(e,q)$-perfect) and $qM^{-1}\in M_{n}(\Z)$. This last fact imply that $M$ is the generator matrix of a code $C \in LPL^{\infty}(n,e,q)$ whose group structure is given by the Smith normal form of $qM^{-1}=A$, that is $C \simeq \Z_{d_1}\times \ldots \Z_{d_n}$.
\end{proof}

We remark that in the $n$-cyclic case, by the structure theorem for finitely generated abelian groups, every abelian group $G$ of order $t^n$ is represented by a code $C \in LPL^{\infty}(n,e,q)$ (in the sense that $C \simeq G$).

\section{Concluding remarks}

In this paper we derive several results on $q$-ary perfect codes in the maximum metric or equivalently about tiling of the torus $\mathcal{T}_q^n$ (or $q$-periodic lattices) by cubes of odd length. A type of matrices (perfect matrices) which provide generator matrices with a special form for perfect codes is introduced. We describe isometry and isomorphism classes of two-dimensional perfect codes extending some results to maximal codes in general dimensions. Several constructions of perfect codes from codes of smaller dimension and via section are given. Through these constructions we extended results obtained for dimension two to arbitrary dimensions and interesting families of $n$-dimensional perfect codes are obtained (as those in Corollaries \ref{CyclicFamilyCor} and \ref{CyclicFamily2Cor}). A characterization of what group isomorphism classes can be represented by $(n,e,q)$-perfect codes is derived for the two-dimensional case, for the maximal case and for the cyclic case.\\


Potential further problems related to this work are interesting to investigate. The fact that every linear perfect code is standard (which is a consequence of Minkoski-Haj{\'o}s theorem) guarantees that the permutation associated to a code (Definition \ref{PermutationAssocDef}) is well defined. It is likely to be possible to extend some of our results to non-linear codes for which the permutation associated to the code is well defined. We also could study isometry classes of perfect non-linear codes (\cite{KP4,MS2} could be helpful). It should be interesting to obtain for higher dimensions a result analogous to the parametrization theorem (Theorem \ref{parametrizationTheorem}), that is, in such a way that isometry and isomorphism classes correspond with certain generalized cosets (Theorems \ref{GralParamTheoremP1} and \ref{GralParamTheoremP3} provide a partial answer for the maximal case). In \cite{Kolountzakis}, tilings by the notched cube and by the extended cube were considered, it may be possible to extended some of results obtained here for these more general shapes. Another remarkable fact is related to properties of the admissible structures (isomorphism classes that can be represented by a perfect code in the maximum metric). It is possible to define a natural poset structure in the set of isomorphism classes of abelian groups of order $t^n$ in such a way that the cyclic group $\Z_{t^n}$ and the cartesian group ${(\Z^{t})}^n$ correspond to the maximum and minimum element in this poset and the admissible structures form an ideal in this poset in specific situations (for example in the two-dimensional case and in the maximal and cyclic case for arbitrary dimensions). We wonder if this last assertion holds in general.


\begin{thebibliography}{99}


\bibitem{AC}
C. Alves, S. Costa, Commutative group codes in $\mathbb{R}^4,\mathbb{R}^6,\mathbb{R}^8$ and $\mathbb{R}^{16}$ - Approaching the bound, Discrete Mathematics 313: 1677-1687, 2013. 

\bibitem{BBV}
M. Blaum, J. Bruck, A. Vardy, Interleaving schemes for multidimensional cluster errors, IEEE transaction on Information Theory 44(2): 730-743, 1998.

\bibitem{CS}
J. H. Conway, N. J. A. Sloane, Sphere Packings, Lattices and Groups, Springer-Verlag, New York, USA, 1998.

\bibitem{CS2}
K. Corr{\'a}di, S. Szab{\'o}, A combinatorial approach for Keller's conjecture, Periodica Mathematica Hungarica 21(2): 95-100, 1990.

\bibitem{EY}
T. Etzion, E. Yaakobi, Error-correction of multidimensional bursts, IEEE transaction on Information Theory 55: 961-976, 2009.

\bibitem{Fraleigh}
J. B. Fraleigh, First Course in Abstract Algebra, Pearson New International Edition, 2013.

\bibitem{GW}
S. W. Golomb, L. R. Welch, Perfect Codes in the Lee metric and the packing of polyominoes, SIAM Journal on Applied Mathematics 18(2): 302-317, 1970.

\bibitem{Hajos}
G. Haj{\'o}s, {\"U}ber einfache und mehrfache Bedeckung des n-dimensionalen Raumes mit einem W{\"u}rfelgitter, Mathematische Zeitschrift 47(1): 427-467, 1942.

\bibitem{Hardy}
G. H. Hardy, Some Famous Problems of the Theory of Numbers and in Particular Waring's Problem, Oxford University Press, UK, 1920.

\bibitem{Keller}
O. H. Keller, {\"U}ber die l{\"u}ckenlose Einf{\"u}llung des Raumes mit W{\"u}rfeln, Journal f{\"u}r die Reine und Angewandte Mathematik 163: 231-248, 1930.

\bibitem{Kisielewicz}
A. P. Kisielewicz, On the structure of cube tilings of $R^3$ and $R^4$, Journal of Combinatorial Theory, series A 120(1): 1-10, 2013.

\bibitem{KLTT10}
T. Klove, T. Lin, D. Tsai, W. Tzeng, Permutation arrays under the Chebyshev distance, IEEE Transaction on Information Theory 56(6): 2611-2617, 2010.

\bibitem{Kolountzakis}
M. N. Kolountzakis, Lattice tilings by cubes: whole, notched and extended, The Electronic Journal of Combinatorics 5(1): $\#$R14, 1998.

\bibitem{KP1}
A. P. Kisielewicz, K. Przeslawski, Polyboxes, cube tilings and rigidity, Discrete \& Computational Geometry 40(1): 1–30, 2008.

\bibitem{KP2}
A. P. Kisielewicz, K. Przeslawski, Rigidity and the chessboard theorem for cube packings, European Journal of Combinatorics 33(6): 1113-1119, 2012.

\bibitem{KP3}
A. P. Kisielewicz, K. Przeslawski, The coin exchange problem and the structure of cube tilings, The Electronic Journal of Combinatorics 19: $\#$R26, 2012.

\bibitem{KP4}
A. Kisielewicz, K. Przeslawski, The structure of cube tilings under symmetry conditions, Discrete \& Computational Geometry 48(3): 777-782, 2012.

\bibitem{Lint}
J. H. van Lint, Introduction to Coding Theory, Springer-Verlag, New York, 1982.

\bibitem{LS}
J. C. Lagarias, P. W. Shor, Keller's cube-tiling conjecture is false in high dimensions, Bulletin of American Mathematical Society 27: 279-283, 1992.

\bibitem{LS2}
J. C. Lagarias, P. W. Shor, Cube tilings of $\mathbb{R}^n$ and nonlinear codes, Discrete \& Computational Geometry 11: 359–391, 1994.

\bibitem{Mackey}
J. Mackey, A cube tiling of dimension eight with no facesharing, Discrete \& Computational Geometry 28(2): 275–279, 2002.

\bibitem{Minkowski}
H. Minkowski, Diophantische Approximationen, Teubner, Leipzig, 1907. Reprint: Physica-Verlag, W{\"u}rzberg, 1961.

\bibitem{MS}
F. J. MacWilliams, N. J. A. Slone, The Theory of Error-Correcting Codes, North-Holland, Amsterdam, 1978.

\bibitem{MS2}
M. Muniz, S. I. R. Costa, Labeling of Lee and Hamming spaces, Discrete Mathematics 260(1): 119-134, 2003.

\bibitem{Perron}
O. Perron, {\"U}ber l{\"u}ckenlose Ausf{\"u}lung des n-dimensionalen Raumes durch kongruente W{\"u}rfel, Mathematische Zeitschrift 46(1): 1-26, 1940.

\bibitem{QC15}
C. Qureshi, S. I. R. Costa, Classification of the perfect codes in the $\infty$-Lee metric, XXXIII Brazilian Telecommunications Symposium (SBrT), 2015.

\bibitem{RS}
R. M. Roth, P. H. Siegel, Lee-metric BCH codes and their application to constrained and partial-response channels, IEEE Transaction on Information Theory 40(4): 1083-1096, 1994.

\bibitem{Schmidt}
K. U. Schmidt, Complementary sets, generalized Reed-Muller codes, and power control for OFDM, IEEE Transaction on Information Theory 53(2): 808-814, 2007.

\bibitem{SI}
M. D. Sikiric, Y. Itoh, Combinatorial cube packings in the cube and the torus, European Journal of Combinatorics 31(2): 517-534, 2010.

\bibitem{SIP}
M. D. Sikiric, Y. Itoh, A. Poyarkov, Cube packings, second moment and holes, European Journal of Combinatorics 28: 715-725, 2007.

\bibitem{ST10}
M. Shieh, S. Tsai, Decoding frequency permutation arrays under Chebyshev distance, IEEE Transaction on Information Theory 56(11): 5730-5737, 2010.

\bibitem{Szabo}
S. Szabó, Topics in Factorization of Abelian Groups, Hindustan Book Agency, New Delhi, India, 2004.

\bibitem{TS10}
I. Tamo, M. Schwartz, Correcting limited-magnitude errors in the rank-modulation scheme, IEEE Transaction on Information Theory 56: 2551-2560, 2010.

\bibitem{Zong}
C. Zong, What is known about unit cubes, Bulletin of American Mathematical Society 42(2): 181-211, 2005.

\end{thebibliography}
\end{document}